\documentclass[12pt]{article}

\usepackage{color,graphicx,amsthm,amsmath,url,afterpage}
\usepackage{graphics,amssymb,latexsym}
\usepackage{helvet}

\usepackage[latin1]{inputenc}
\usepackage{fancyhdr}
\usepackage{eurosym}
\usepackage{lastpage}
\usepackage{tabularx}
\usepackage{enumitem}
\usepackage{epsfig}
\usepackage{setspace}
\usepackage{wrapfig}
\usepackage{multicol}
\usepackage{multirow}
\usepackage{colortbl}
\usepackage{hyperref}

\newtheorem{Th}{Theorem}
\newtheorem{Rem}{Remark}
\newtheorem{Ex}{Example}
\newtheorem{Cor}{Corollary}

\title{Strictly self-similar fractals composed of star-polygons that are attractors of Iterated Function Systems.}  
\author{Vassil Tzanov}

\begin{document}

\maketitle

\begin{abstract}
 In this paper are investigated strictly self-similar fractals that are composed of an infinite number of regular star-polygons, also known as Sierpinski $n$-gons, $n$-flakes or polyflakes. 
Construction scheme for Sierpinsky $n$-gon and $n$-flake is presented where the dimensions of the Sierpinsky $\infty$-gon and the $\infty$-flake are computed to be 1 and 2, respectively.
These fractals are put in a general context and Iterated Function Systems are applied for the visualisation of the geometric iterations of the initial polygons, as well as the visualisation of sets of points that lie on the attractors of the IFS generated by random walks. 
Moreover, it is shown that well known fractals represent isolated cases of the presented generalisation.
The IFS programming code is given, hence it can be used for further investigations.
\end{abstract}
\newpage

\section{Introduction - the Cantor set and regular star-polygonal attractors}

A classic example of a strictly self-similar fractal that can be constructed by Iterated Function System is the Cantor Set \cite{Mandelbrot_1982}. 
Let us have the interval $E=[-1,1]$ and the contracting maps $S_1,S_2:$ $\mathbb{R}\to \mathbb{R}$, $S_1(x)=x/3-2/3$, $S_2=x/3+2/3$, where $x\in E$. Also $S^k:$ $S(S^{k-1}(E))=S^k(E)$, $S^0(E)=E$, where $S(E)=S_1(E)\cup S_2(E)$.
Thus if we iterate the map $S$ infinitely many times this will result in the well known Cantor Set; see figure \ref{fr1}.
This iteration procedure can be generalised by the following theorem \cite{Falconer_1990}:
\begin{Th}\label{th1} 
If we have $S_1,...,S_N:$ $|S_i(x)-S_i(y)|\leq c_i|x-y|$, $c_i<1$, then $\exists$ unique non-empty set $F:$ $F=\cup_{i=1}^{N}S_i(F)$, hence invariant for the map $S$ and $F=\cap_{k=1}^{\infty}S^k(E)$ 
\end{Th}  
\begin{figure}[hbp]
\begin{center}
\begin{picture}(140,70)(0,0)
\put(0,0){
\put(0,0){\includegraphics{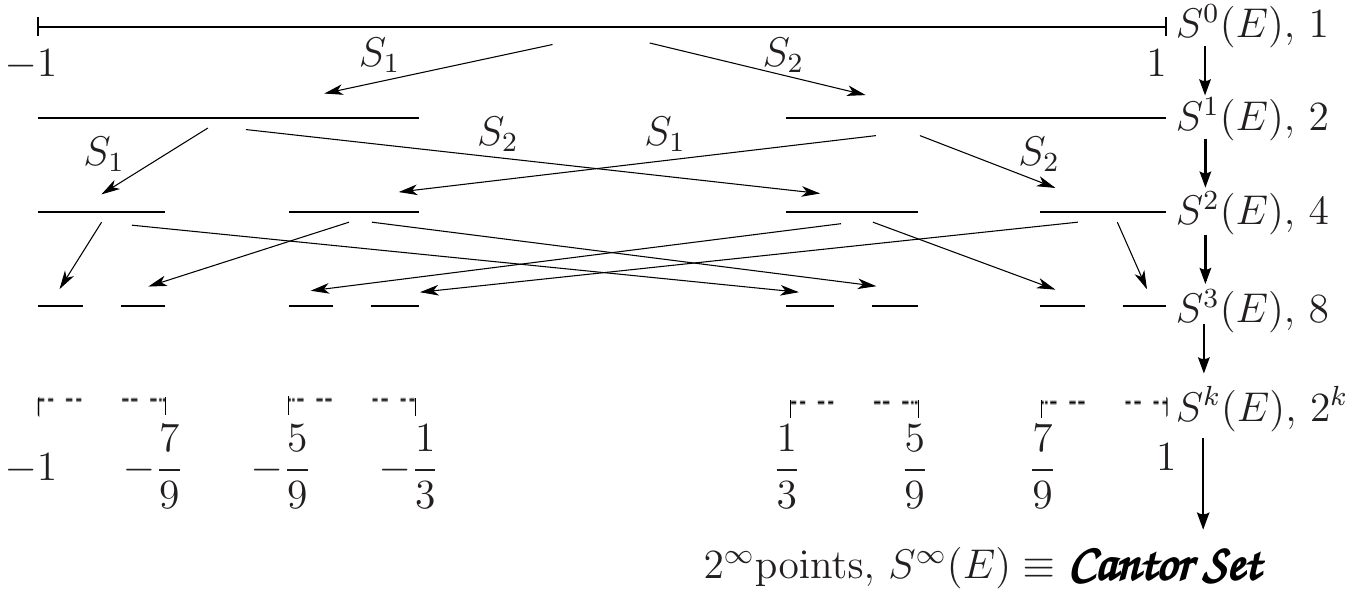}}
}
\end{picture}
\end{center}
\vspace{0.0cm}
\caption{Sketch of the repeated actions of the maps $S_1$ and $S_2$ on the interval $E$ that result in the Canator Set, where the left and right arrows represent $S_1$ and $S_2$, respectively. 
On the right hand side of the intervals, the corresponding iterations of the map $S$ and the number of the intervals of the particular iteration are shown.}\label{fr1}
\end{figure}
Using a polygon as an initial set for fractal generation is a well known technique since the most famous strictly self-similar fractal examples, the Cantor set and the Sierpinski triangle, consist of infinitely many line-segments and triangles, respectively.
In the present paper the number of the vertices of the polygon will be increased to an arbitrary number $n\in\mathbb{N}, n\geq 2$.
Thus, the fractal can consist of infinitely many pentagrams, hexagons, etc.  
Furthermore, the building regular polygons will be all $\{n/m\}$ star-polygons \cite{Coxeter_1947} where $n\geqslant 2$ and $m\leqslant n/2$, $n\in\mathbb{N}$, and $m\in\mathbb{N}$.
For our purpose we will take the unit circle and it will be divided in $n$ equal segments. 
For example, in the case of pentagram, we have $m=2$ which gives us $\{5/2\}$-polygon.   
Once we choose an $\{n/m\}$-polygon it can be scaled by a factor of $P\in(0,1)$ with respect to all of the vertices of the polygon.
This will produce $n$ new polygons similar to the initial one but scaled down by factor of $P$.
Now, if we repeat the procedure for each one of these new polygons, another $n^2$ polygons will be created that will be $P^2$ times smaller than the initial $\{n/m\}$-polygon inscribed in the unit circle.  
If the $P$ is chosen carefully, after infinitely many contractions, the result will be a strictly self-similar fractal composed by non-intersecting polygons.  
Thus, at the $i$-th contraction the defined $n$ contracting maps are applied $n^i$ times, and by theorem \ref{th1} when $i\rightarrow\infty$ we will define infinitely many points of attraction.
These points of attraction specify the attractor of the iteration procedure.
This polygonal attractor is a fractal produced by the infinite contractions ($n^i$, where $i\rightarrow\infty$) of the initial polygon and it is self similar, i.e. it is composed of infinitely many polygons similar to the initial one.   

The present study is focused on non-self-intersecting fractals where the scaled copies of the initial polygons osculate with each other.
This restricts the possibilities for $P$ when Sierpinsky $n$-gon \cite{Dennis_1995} for arbitrary $n$ is constructed and a formula about $P$ is derived.
This important ratio is reported in two other places \cite{Dennis_1995} and \cite{Zuhair_2009} where in the latter the proof is omitted.
Moreover, in both manuscripts the authors did not prove that the  Hausdorff dimension of the non-self-intersecting $\infty$-gon is 1. 
In the present paper, an original derivation of the equation for the scaling ratio $P$ is presented.
It is done in a very detailed way by using simple geometric laws which makes the result affordable even for high-school students.
Also, the Hausdorff dimension of the $\infty$-gon is shown to be 1.
Furthermore, universal constructions for $n$-flakes are proposed for the cases when $n$ is even or odd, and the Hausdorff dimension of the $\infty$-flake is proved to be 2.
To this end, formulas for the scaling ratio and the rotation of the central polygon of the $n$-flake are derived, which to the knowledge of the author have not been reported previously.  
Finally, it is shown that different initial polygons may result in an identical attractor when an IFS iteration scheme is applied and it is shown which are the main parameters that define the shape of the Sierpinski $n$-gon and the polyflake attractors.

The paper is constructed as follows: in section \ref{sec2} the parameters $P$ and $m$ are introduced and an important equation for the ratio $P=P(n,m)$ is derived.
In section \ref{sec3} a condition for $m$ is obtained that ensures no self-intersection of the studied class of fractals. 
Then, two techniques for imaging IFS attractors are introduced and several Sierpinsky $n$-gons are computed together with their dimensions. 
In section \ref{sec4} the possibility for additional scaling map that scales down towards the centre of the polygon is taken into account. 
Different constructions for the cases when $n$ is odd or even are proposed so the resulted $n$-flake to be non-self-intersecting for arbitrary $n$. 
Also, a few interesting examples are given and the Hausdorff dimension of the $\infty$-flake is computed.
In section \ref{sec5} some well known fractals are shown to be a special case of the fractal generation scheme shown here.
It is also explained why identical attractors may originate from different star-polygons and how we can exploit this feature.  

\newpage 
\section{The parameters $P$ and $m$}\label{sec2}
An important result of the present paper will be explained in this section. 
Here the scaling parameter $P$ will be deduced from $n$ and $m$.
Therefore, $P=P(n,m)$ is a specific scaling factor for the chosen initial $\{n/m\}$-polygon, where $P$ does not depend on the diameter of the circumscribed circle.

In figure \ref{fr2} a sketch of a $\{n/3\}$ star-polygon is shown where $m=3$. 
Here, the vertices $A_i$ for $i=1,...,n$ are the vertices of the $\{n/3\}$ star-polygon and $O_a$ is the centre of the circumscribed circle $\cal{S_{\text{a}}}$, $M$ is the intersection point of the secants $A_1A_4$ and $A_3A_6$, $H$ is the orthogonal projection of $O_a$ on $A_1A_4$ and $L$ is the orthogonal projection of $O_a$ on $A_3A_4$.
Our purpose will be to find the ratio $P=\displaystyle\frac{MA_4}{A_1A_4}$ because $MA_4$ will be a line segment of the star-polygon resulted after the scaling of the initial polygon with respect to the point $A_4$ by factor of $P$.

\begin{figure}[t]
\begin{center}
\begin{picture}(140,70)(0,0)
\put(0,0){
\put(0,0){\includegraphics{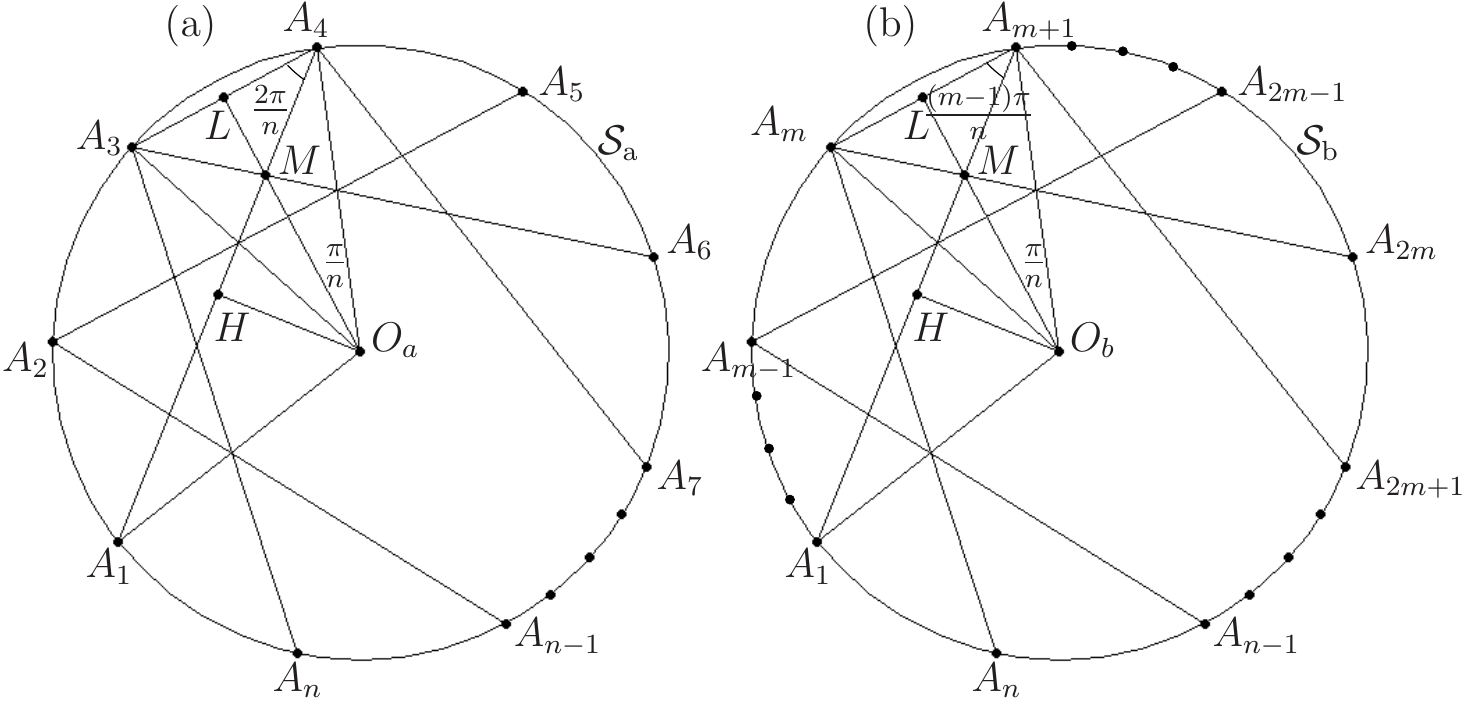}}
}
\end{picture}
\end{center}
\vspace{-0.5cm}
\caption{Sketches of \{n/3\} and \{n/m\}  star-polygons circumscribed in $\cal{S_{\text{a}}}$ and $\cal{S_{\text{b}}}$ respectively.}\label{fr2}
\end{figure}

 The initial polygon is circumscribed by $\cal{S_{\text{a}}}$ and $O_a$ is its centre thus $\measuredangle{A_3O_aA_4}=2\pi/n \Rightarrow \measuredangle{LO_aA_4}=\pi/n$ because $A_3O_a=A_4O_a$. 
Also, $\measuredangle{A_1O_aA_3}=4\pi/n \Rightarrow \measuredangle{A_1A_4A_3}=2\pi/n$ and $\measuredangle{A_1O_aH}=3\pi/n$, since $\cal{S_{\text{a}}}$ with centre at $O_a$ circumscribes $A_1$, $A_3$ and $A_4$.
Now we can deduce $A_1A_4$ and $A_4M$ by the radius $r$ of $\cal{S_{\text{a}}}$; $r=O_aA_i$ for $i=1,...,n$ : 

$A_1A_4=2r\sin(3\pi/n)$. 
In order to deduce $A_4M$ we will first find $A_4L$. 
Thus, $A_4L=r\sin(\pi/n) \Rightarrow A_4M=\displaystyle\frac{A_4L}{\cos(2\pi/n)}=\displaystyle\frac{r\sin(\pi/n)}{\cos(2\pi/n)}$.
Now we can substitute the values for $A_1A_4$ and $MA_4$ in order to find $P=\displaystyle\frac{MA_4}{A_1A_4}=\displaystyle\frac{\sin(\pi/n)}{2\cos(2\pi/n)\sin(3\pi/n)}$.

We have just found $P(n,3)$, now, let us do the same computations for any $1\leq m \leq n/2$. 
In figure \ref{fr2}(b) we can find the sketch of a $\{n/m\}$ star-polygon where the vertices $A_i$ for $i=1,...,n$ are the vertices of the star-polygon and $O_b$ is the centre of the circumscribed circle $\cal{S_{\text{b}}}$ with radius $r$, $M$ is the intersection point of the secants $A_1A_{m+1}$ and $A_{m}A_{2m}$, $H$ is the orthogonal projection of $O_b$ on $A_1A_{m+1}$ and $L$ is the orthogonal projection of $O_b$ on $A_{m}A_{m+1}$. 
Our purpose will be to find the ratio $P=\displaystyle\frac{MA_{m+1}}{A_1A_{m+1}}$.
Thus, $\measuredangle{LO_bA_{m+1}}=\pi/n$ because $A_{m}O_b=A_{m+1}O_b$, $\measuredangle{A_1O_bA_{m}}=(2m-2)\pi/n \Rightarrow \measuredangle{A_1A_{m+1}A_{m}}=(m-1)\pi/n$ and $\measuredangle{A_1O_bH}=m\pi/n$, since $\cal{S_{\text{b}}}$ with centre at $O_b$ circumscribes $A_1$, $A_{m}$ and $A_{m+1}$.
Now we can deduce $A_1A_{m+1}$ and $A_{m+1}M$ by the radius $r$ of $\cal{S_{\text{b}}}$: 
$A_1A_{m+1}=2r\sin(m\pi/n)$. 
In order to deduce $A_{m+1}M$ we will first find $A_{m+1}L$. 
Thus, $A_{m+1}L=r\sin(\pi/n) \Rightarrow A_{m+1}M=\displaystyle\frac{A_{m+1}L}{\cos((m-1)\pi/n)}=\displaystyle\frac{r\sin(\pi/n)}{\cos((m-1)\pi/n)}$
Now we can substitute the values for $A_1A_{m+1}$ and $MA_{m+1}$ in order to find:\\
\footnotesize\begin{equation}\label{eq1}
P=\displaystyle\frac{MA_{m+1}}{A_1A_{m+1}}=\displaystyle\frac{\sin(\pi/n)}{2\cos((m-1)\pi/n)\sin(m\pi/n)}
\end{equation}
\normalsize
\begin{figure}[t]
\begin{center}
\begin{picture}(140,75)(0,0)
\put(0,0){
\put(-10,0){\includegraphics{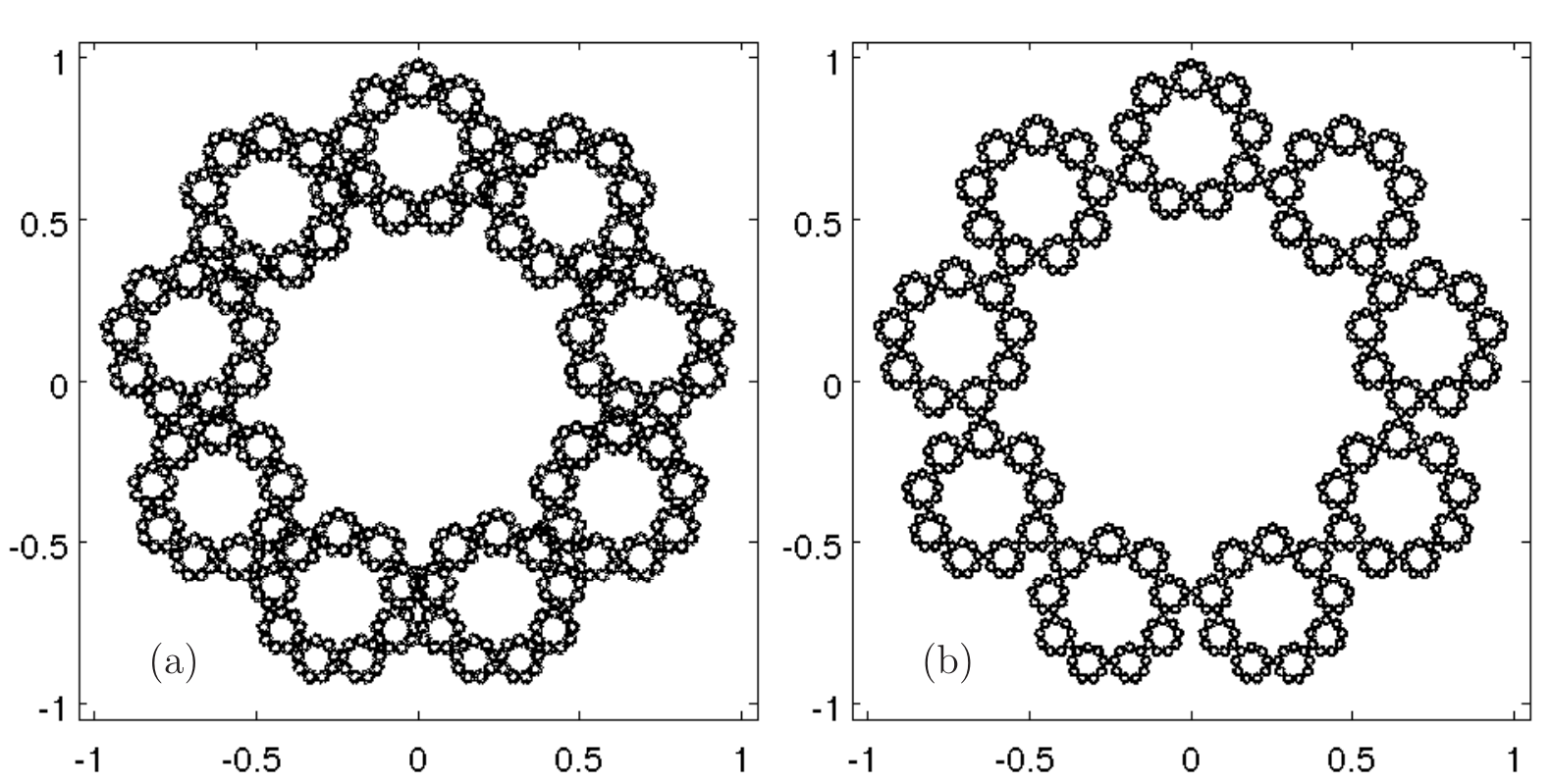}}
}
\end{picture}
\end{center}
\vspace{-0.5cm}
\caption{IFS generated fractal sets made of points that lie on \{9/2\} and \{9/3\} star-polygons in subpanels (a) and (b), respectively.}\label{fr3}
\end{figure}


\section{Generation of fractals by using IFS}\label{sec3}
The values for $P$ obtained in the previous section will be used here for the computation of self-similar fractals that are IFS attractors.
These attractors will be derived by the random walk/orbit method or so called chaos game \cite{Devaney_1990,Peitgen_2004}.
First, we define a matrix that specifies how many points along the unit circle will be taken into account ($n$), and what is the contraction $P(n,m)$ towards those points.
For example, the matrices for the fractals in figures \ref{fr3}(a) and \ref{fr3}(b) are bellow; see Table \ref{T1}.

\begin{table}
\caption{The two matrices $M_{\{9/2\}}$ and $M_{\{9/3\}}$ where each of them defines nine contracting maps (every two rows define a map), needed for the random IFS procedure, resulting in the \{9/2\} and \{9/3\} fractal attractors shown in figure \ref{fr3}(a) and \ref{fr3}(b).}
\begin{center}
\footnotesize
\begin{tabular}{|p{1cm}|p{1cm}|p{1.5cm}||p{1cm}|p{1cm}|p{1.5cm}|}
\hline \multicolumn{3}{|c||}{$M_{\{9/2\}}$} & \multicolumn{3}{|c|}{$M_{\{9/3\}}$} \\
\hline   0.2831 & 0 & 0.8660254 & 0.2578 & 0 & 0.8660254\\
\hline   0 & 0.2831 & -0.5 & 0 & 0.2578 & -0.5\\
\hline   0.2831 & 0 & 0.9848078 & 0.2578 & 0 & 0.9848078\\
\hline   0 & 0.2831 & 0.1736482 & 0 & 0.2578 & 0.1736482\\
\hline   0.2831 & 0 & 0.6427876 & 0.2578 & 0 & 0.6427876\\
\hline   0 & 0.2831 & 0.7660444 & 0 & 0.2578 & 0.7660444\\
\hline   0.2831 & 0 & 0 & 0.2578 & 0 & 0\\
\hline   0 & 0.2831 & 1 & 0 & 0.2578 & 1\\
\hline   0.2831 & 0 & -0.642788 & 0.2578 & 0 & -0.642788\\
\hline   0 & 0.2831 & 0.7660444 & 0 & 0.2578 & 0.7660444\\
\hline   0.2831 & 0 & -0.984808 & 0.2578 & 0 & -0.984808\\
\hline   0 & 0.2831 & 0.1736482 & 0 & 0.2578 & 0.1736482\\
\hline   0.2831 & 0 & -0.866025 & 0.2578 & 0 & -0.866025\\
\hline   0 & 0.2831 & -0.5 & 0 & 0.2578 & -0.5\\
\hline   0.2831 & 0 & -0.34202 & 0.2578 & 0 & -0.34202\\
\hline   0 & 0.2831 & -0.939693 & 0 & 0.2578 & -0.939693\\
\hline   0.2831 & 0 & 0.3420201 & 0.2578 & 0 & 0.3420201\\
\hline   0 & 0.2831 & -0.939693 & 0 & 0.2578 & -0.939693\\
\hline
\end{tabular}
\normalsize
\end{center}
\end{table}\label{T1}
\vspace{0.4cm}

This matrix is then plugged into the random generator, where the number of points that we want to map over the IFS attractor are specified; see the Appendix section \ref{app} for the MATLAB code.
\subsection{Condition for non-self-intersection}
In figure \ref{fr3} two examples of star-polygon fractals with initial $\{9/2\}$- and $\{9/3\}$-polygons clarify why the parameter $m$ in the ratio $\displaystyle\frac{MA_{m+1}}{A_1A_{m+1}}$ is important when non-self-intersecting fractals are desired.
We would like the self-intersection of the resulted sets to be prevented, thus, we will state the following theorem.

\begin{Th}\label{V1}
If we have a strictly self-similar fractal set obtained as an attractor of IFS, where $n$-attracting points lie on $S^1$, so that they are the vertices of a $\{n/m\}$ star-polygon, and where the attraction towards these points is $P=P(n,m)$ given by equation (\ref{eq1}), then this fractal set is not-self-intersecting if and only if $m\in[n/4,n/4+1]$, which uniquely defines $P$ for a given $n$.
\end{Th}

\begin{proof}
For clarity, one must look at figure \ref{fr4} where with red is denoted the scaled down polygon towards $A_{m+1}$, self-similar to the original one.
Although, it has $9$ vertices, it must be considered as ${n/m}$ star-polygon, because we will only use geometrical properties that are independent of $n$ and $m$. 
For this purpose we must find out the following angles: $\measuredangle{A'MA_{m+1}}$ and $\measuredangle{O_bMA_{m+1}}.$ 
\begin{figure}
\begin{center}
\begin{picture}(140,80)(0,0)
\put(0,0){
\put(0,0){\includegraphics{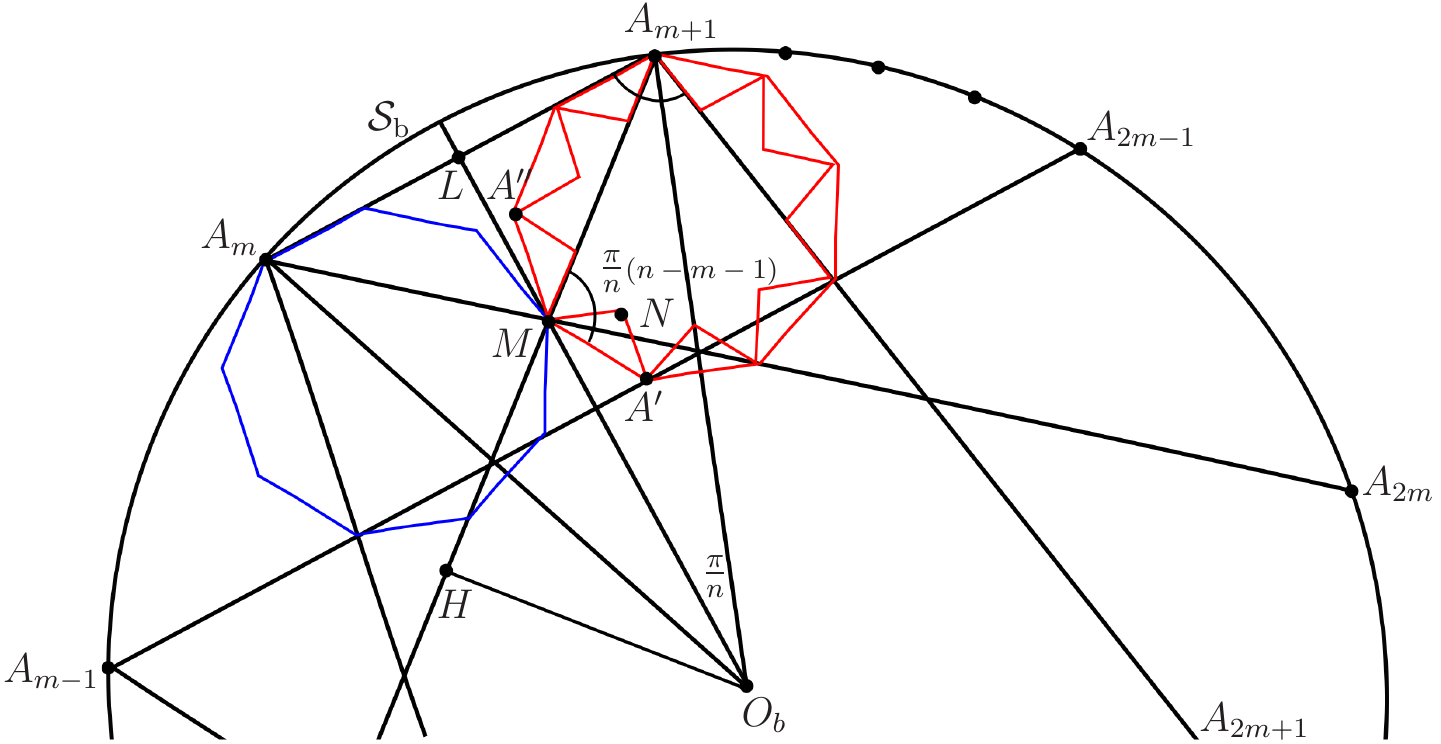}}
}
\end{picture}
\end{center}
\vspace{-0.5cm}
\caption{Sketch of \{n/m\} star-polygon circumscribed in $\cal{S_{\text{b}}}$ together with the scaled down polygon towards the $A_{m+1}$ vertex (in red) and its mirror image across the line $O_bL$ (in blue) which is equivalent to the scaled down polygon towards the $A_{m}$ vertex.}\label{fr4}
\end{figure}

We already found that $\measuredangle{LA_{m+1}M}=(m-1)\pi/n$ and therefore\\
$\measuredangle{MA_{m+1}O_{b}}=\measuredangle{O_bA_{m+1}A_{2m+1}}=\pi/2-\pi/n-(m-1)\pi/n$ $\Rightarrow$ \\
$\Rightarrow$ $\measuredangle{LA_{m+1}A_{2m+1}}=2(\pi/2-\pi/n-(m-1)\pi/n)+(m-1)\pi/n=\pi-2\pi/n-(m-1)\pi/n$ $\Rightarrow$\\
$\Rightarrow$ $\measuredangle{LA_{m+1}A_{2m+1}}=\displaystyle\frac{\pi}{n}(n-m-1)$.
On the other hand due to the symmetry of the scaled polygon (in red) $\measuredangle{LA_{m+1}A_{2m+1}}=\measuredangle{A_{m+1}MA'}$ since  
$\measuredangle{A_{m+1}MN}=\measuredangle{MA_{m+1}A_{2m+1}}$ and $\measuredangle{NMA'}=\measuredangle{MA_{m+1}L}$.
$\measuredangle{O_bMA_{m+1}}=\pi/2+(m-1)\pi/n$ as an exterior angle for $\triangle MLA_{m+1}$.

Now, note that the scaled polygon towards the $A_{m+1}$ vertex (in red) reflected across the $O_bL$ line segment will result in the scaled polygon towards the $A_{m}$ vertex (in blue); see also figure \ref{fr2}. 
This is true because the $O_bL$ represents an axis of symmetry for the initial polygon and it contains point $M$ which is a common point for both the scaled polygons towards the $A_{m}$ and $A_{m+1}$ vertices. 
Now we can deduce that these two scaled polygons will not intersect with each other if the vertices $A'$ and $A''$ stay together with $A_{m+1}$ on the same side with respect to the $O_bL$ axis. 
Then, $\measuredangle{O_bMA_{m+1}}\geq\measuredangle{A_{m+1}MA'}$ and $\measuredangle{LMA_{m+1}}\geq\measuredangle{A_{m+1}MA''}$ $\Rightarrow$\\
\vspace{-0.4cm}
\footnotesize
\begin{align}\label{eq2}
\measuredangle{O_bMA_{m+1}}\geq\measuredangle{A_{m+1}MA'} ~~~&~~~ \measuredangle{LMA_{m+1}}\geq\measuredangle{A_{m+1}MA''}\\
\pi/2+(m-1)\pi/n\geq\pi-2\pi/n-(m-1)\pi/n ~~~&~~~ \pi/2-(m-1)\pi/n\geq (m-1)\pi/n \nonumber \\
(4m-4)/2n\geq (n-4)/2n ~~~&~~~ 1/2\geq 2(m-1)/n \nonumber \\
m\geq n/4 ~~~&~~~ n/4+1 \geq m \nonumber
\end{align}
\normalsize
Thus, if the resulted fractal does not self-intersect, then $n/4\geq m\geq n/4+1$, which uniquely defines $m$ except when $n$ can be divided by $4$ without residual. 
At the same time, if the strict inequalities of equations (\ref{eq2}) hold, this ensures that both of the vertices $A'$ and $A''$ stay on the right hand side of $O_bL$ line segment (see figure \ref{fr4}) $\Rightarrow$ their mirror images with respect to $O_bL$ will stay on the left hand side of $O_bL$.
Otherwise, if $A'$ and $A''$ were to cross the $O_bL$ segment becoming on its left hand side, the two polygons would intersect with each other and this intersection would be repeated everywhere since the resulted fractal is strictly self-similar.
Finally, if $A'$ or $A''$ lies on $O_bL$ and the two scaled polygon have a common side $MA'$ or $MA''$, then one of the equations (\ref{eq2}) must be with a sign for equality.
\footnotesize
\begin{align}\label{eq3}
P=\frac{\sin(\frac{\pi}{n})}{2\cos(\frac{(m-1)\pi}{n})\sin(\frac{m\pi}{n})} \nonumber \\
P=P_1(n,m)~ \text{for} ~n=4v~ \text{and} ~m=v ~~&~~ P=P_2(n,m)~ \text{for} ~n=4v~ \text{and} ~m=v+1 \nonumber \\
P_1=\frac{\sin(\frac{\pi}{4v})}{2\cos(\frac{(v-1)\pi}{4v})\sin(\frac{v\pi}{4v})} ~~&~~ P_2=\frac{\sin(\frac{\pi}{4v})}{2\cos(\frac{v\pi}{4v})\sin(\frac{(v+1)\pi}{4v})} \nonumber \\
P_1=\frac{\sin(\frac{\pi}{4v})}{2\cos(\frac{(v-1)\pi}{4v})\sin(\frac{\pi}{4})} ~~&~~ P_2=\frac{\sin(\frac{\pi}{4v})}{2\cos(\frac{\pi}{4})\sin(\frac{(v+1)\pi}{4v})} \nonumber \\
P_1=\frac{\sin(\frac{\pi}{4v})}{\sqrt{2}\cos(\frac{(v-1)\pi}{4v})} ~~&~~ P_2=\frac{\sin(\frac{\pi}{4v})}{\sqrt{2}\sin(\frac{(v+1)\pi}{4v})} \nonumber \\
P_1=\frac{\sin(\frac{\pi}{4v})}{\sqrt{2}\cos(\frac{v\pi}{4v})\cos(\frac{\pi}{4v})+\sqrt{2}\sin(\frac{v\pi}{4v})\sin(\frac{\pi}{4v})} ~~&~~ P_2=\frac{\sin(\frac{\pi}{4v})}{\sqrt{2}\sin(\frac{v\pi}{4v})\cos(\frac{\pi}{4v})+\sqrt{2}\cos(\frac{v\pi}{4v})\sin(\frac{\pi}{4v})} \nonumber \\
P_1=\frac{\sin(\frac{\pi}{4v})}{\cos(\frac{\pi}{4v})+\sin(\frac{\pi}{4v})} ~~&~~ P_2=\frac{\sin(\frac{\pi}{4v})}{\cos(\frac{\pi}{4v})+\sin(\frac{\pi}{4v})}  
\end{align}
\normalsize 
Therefore, $n$ must be divided by $4$ without residual because $m\in\mathbb{N}$.
Thus, when $n$ is divided by $4$ without residual and the scaling ratio is $P=P(n,m)$ from equations (\ref{eq2}), then each of the scaled polygons has two common vertices with both of the adjacent scaled polygons. 

The case when $n$ is divided by $4$ without residual and $m\in[n/4,n/4+1]$ implies that equation (\ref{eq1}) produces two values for $P$.
We will compute those values $P(4v,v)$ and $P(4v,v+1)$, where $n=4v$ for some $v>0, v\in\mathbb{N}$; see equations (\ref{eq3}). 
Equations (\ref{eq3}) clearly show that $P(n,n/4)=P(n,n/4+1)$, when $4$ divides $n$ without residual, which ensures the unique definition of $P(n,m)$ when the resulted fractals are non-self-intersecting. 
\end{proof}

\begin{Cor}
The attractors of the IFS with $n$ attracting points that are the vertices of \{n,n/4\} and \{n,n/4+1\} star-polygons and whose scaling ratios are P(n,n/4) and P(n,n/4+1), respectively, are identical and composed of star-polygons that have two vertices in common with each of the adjacent polygons.
\end{Cor}

\begin{Rem}
Note that theorem \ref{V1} is stated for the specific type of self-similar fractal sets obtained using IFS where the attracting points are the vertices of a $\{n/m\}$ star-polygon and $P$ is defined by $n$ and $m$ as given in equation (\ref{eq1}).
Thus, it does not exclude other kinds of non-self-intersecting star-polygonal fractal sets constructed in a different fashion.
\end{Rem}

\subsection{Fractal dimensions}\label{dimensions}
Theorem \ref{V1} allow us to define $F_{\{n/m\}}$ as the IFS attractor produced from an initial equilateral \{n/m\}-polygon with $n$ contracting maps that scale towards the $n$-vertices of the initial \{n/m\}-polygon with ratio $P(n,m)$. We ensured that the obtained self-similar fractal set $F_{\{n/m\}}$ is non-self-intersecting, which allows us to compute its Hausdorff dimension $dim_HF_{\{n/m\}}$ \cite{Falconer_1990} by solving the following equation:
\begin{equation}\label{eq4}
 \sum_{i=1}^Nc_i^{dim_HF_{\{n/m\}}}=1
\end{equation}
where $N$ indicates the amount of similarity maps $S_i$ (see Theorem \ref{th1}) and $0<c_i<1$ are the scaling ratios for each similarity. 

Since $P(2,1)=1/2$, $dim_HF_{\{2/1\}}=-ln(2)/ln(P(2,1))=1$ which means that the attractor of $F_{\{2/1\}}$ is the initial line segment that connects both vertices or one can think about the Cantor set with scale ratio $1/2$.

For the $F_{\{9/3\}}$ in figure \ref{fr3}(b) we have $9P(9,3)^{dim_HF_{\{9/3\}}}=1$ which lead to $dim_HF_{\{9/3\}}=-ln(9)/ln(P(9,3))\approx1.6207585335597825$. 

In the following figures \ref{fr5}, \ref{fr6}, \ref{fr7} and \ref{fr8} are shown the attractors of $F_{\{n/m\}}$ where $n=3,4,5,6,7,8,10,11,12,13,14,15,16,24$ and $m\in[n/4,n/4+1]$. 
The dimensions of the presented fractals are computed in the examples that follows every figure.
One can recognize well known fractals in the cases of $n=3,4,5,6$ but the other examples are not that famous due to the need of a special ratio in order to be constructed.

\newpage
\begin{figure}[hbp]
\begin{center}
\begin{picture}(140,160)(0,0)
\put(0,0){
\put(-10,0){\includegraphics{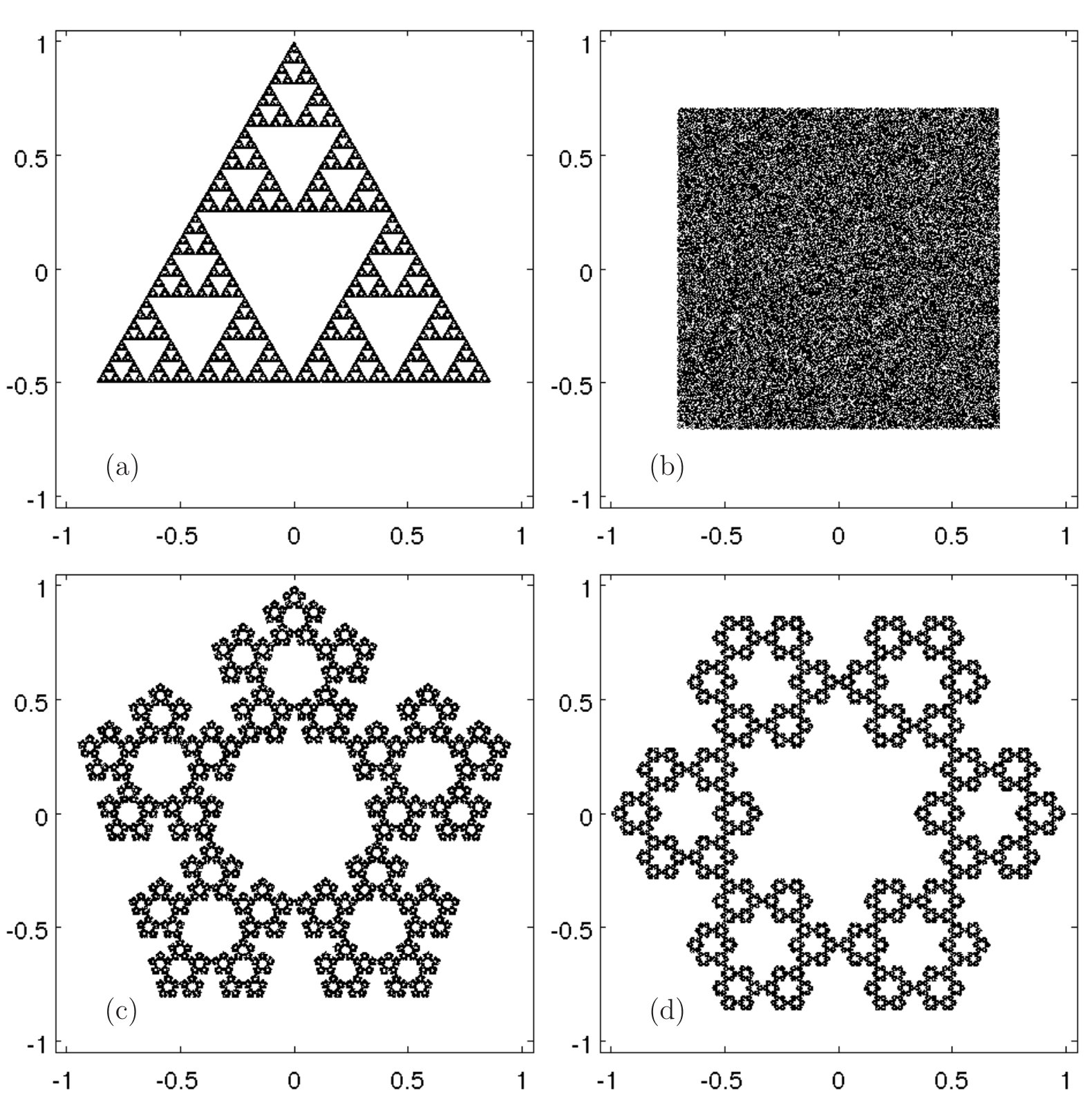}}
}
\end{picture}
\end{center}
\vspace{-0.5cm}
\caption{IFS generated fractal sets made of points that lie in $F_{\{3/1\}}$, $F_{\{4/2\}}$, $F_{\{5/2\}}$ and $F_{\{6/2\}}$ in subpanels (a), (b), (c) and (d), respectively.}\label{fr5}
\end{figure}
\begin{Ex}\label{ex1}
 We will compute the Hausdorff dimensions of the attractors shown in figure \ref{fr5}:
\begin{itemize}[noitemsep, topsep=0pt]
\item Figure \ref{fr5}(a) $dim_HF_{\{3/1\}}=-ln(3)/ln(P(3,1))\approx1.5849625007211563$ 
\item Figure \ref{fr5}(b) $dim_HF_{\{4/2\}}=-ln(4)/ln(P(4,2))=2$ 
\item Figure \ref{fr5}(c) $dim_HF_{\{5/2\}}=-ln(5)/ln(P(5,2))\approx1.6722759381845547$ 
\item Figure \ref{fr5}(d) $dim_HF_{\{6/2\}}=-ln(6)/ln(P(6,2))\approx1.6309297535714573$ 
\end{itemize}
\end{Ex}

\newpage
\begin{figure}[hbp]
\begin{center}
\begin{picture}(140,160)(0,0)
\put(0,0){
\put(-10,0){\includegraphics{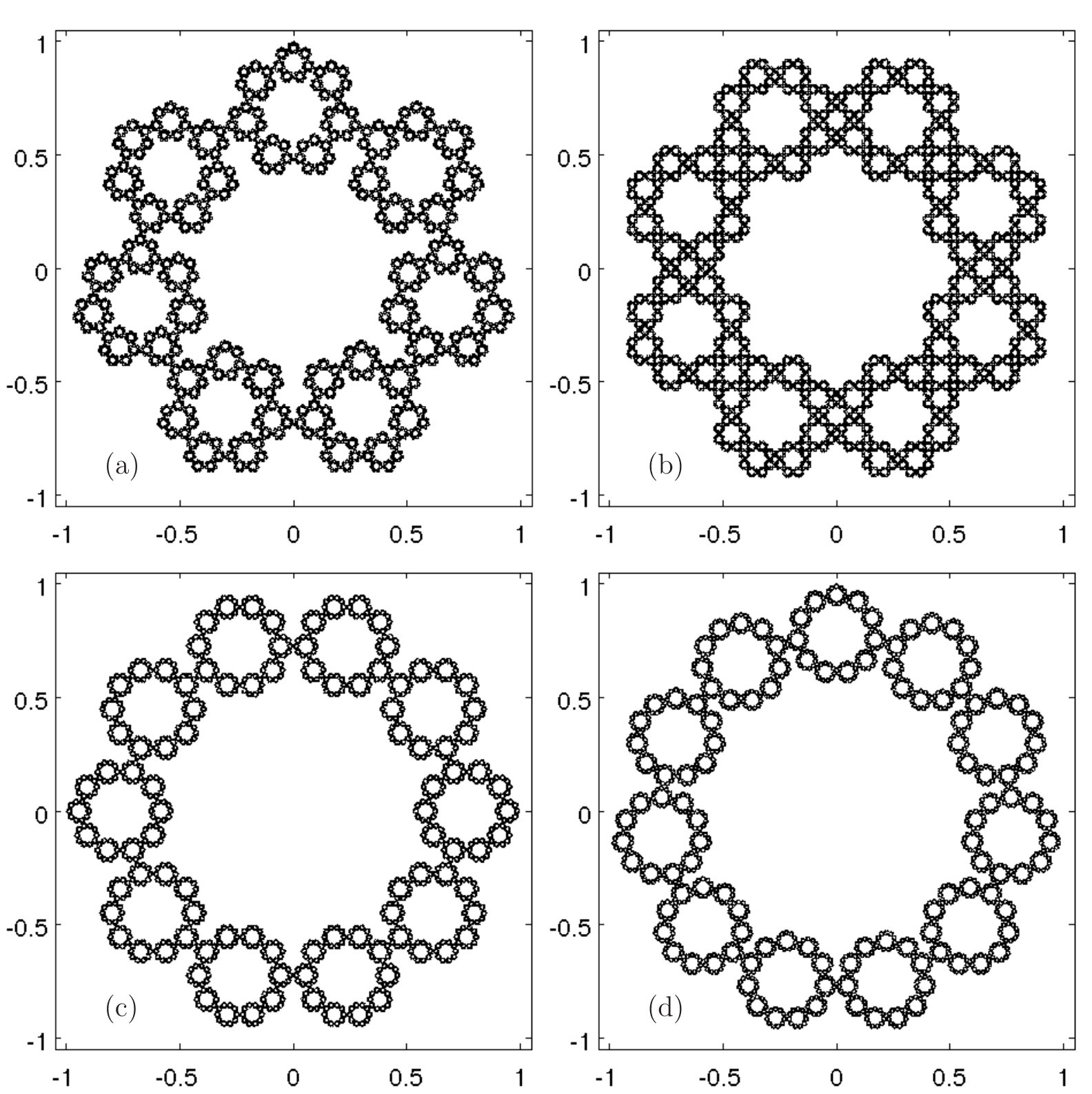}}
}
\end{picture}
\end{center}
\vspace{-0.5cm}
\caption{IFS generated fractal sets made of points that lie in $F_{\{7/2\}}$, $F_{\{8/3\}}$, $F_{\{10/3\}}$ and $F_{\{11/3\}}$ in subpanels (a), (b), (c) and (d), respectively. The IFS attractor of $F_{\{9/3\}}$ is shown in figure \ref{fr3}(b).}\label{fr6}
\end{figure}
\begin{Ex}\label{ex2}
 We will compute the Hausdorff dimensions of the attractors shown in figure \ref{fr6}: 
\begin{itemize}[noitemsep, topsep=0pt]
\item Figure \ref{fr6}(a) $dim_HF_{\{7/2\}}=-ln(7)/ln(P(7,2))\approx1.6522616056918107$ 
\item Figure \ref{fr6}(b) $dim_HF_{\{8/3\}}=-ln(8)/ln(P(8,3))\approx1.6934291475411138$ 
\item Figure \ref{fr6}(c) $dim_HF_{\{10/3\}}=-ln(10)/ln(P(10,3))\approx1.5949906555938886$ 
\item Figure \ref{fr6}(d) $dim_HF_{\{11/3\}}=-ln(11)/ln(P(11,3))\approx1.5911325154416658$ 
\end{itemize}
\end{Ex}

\newpage
\begin{figure}[hbp]
\begin{center}
\begin{picture}(140,160)(0,0)
\put(0,0){
\put(-10,0){\includegraphics{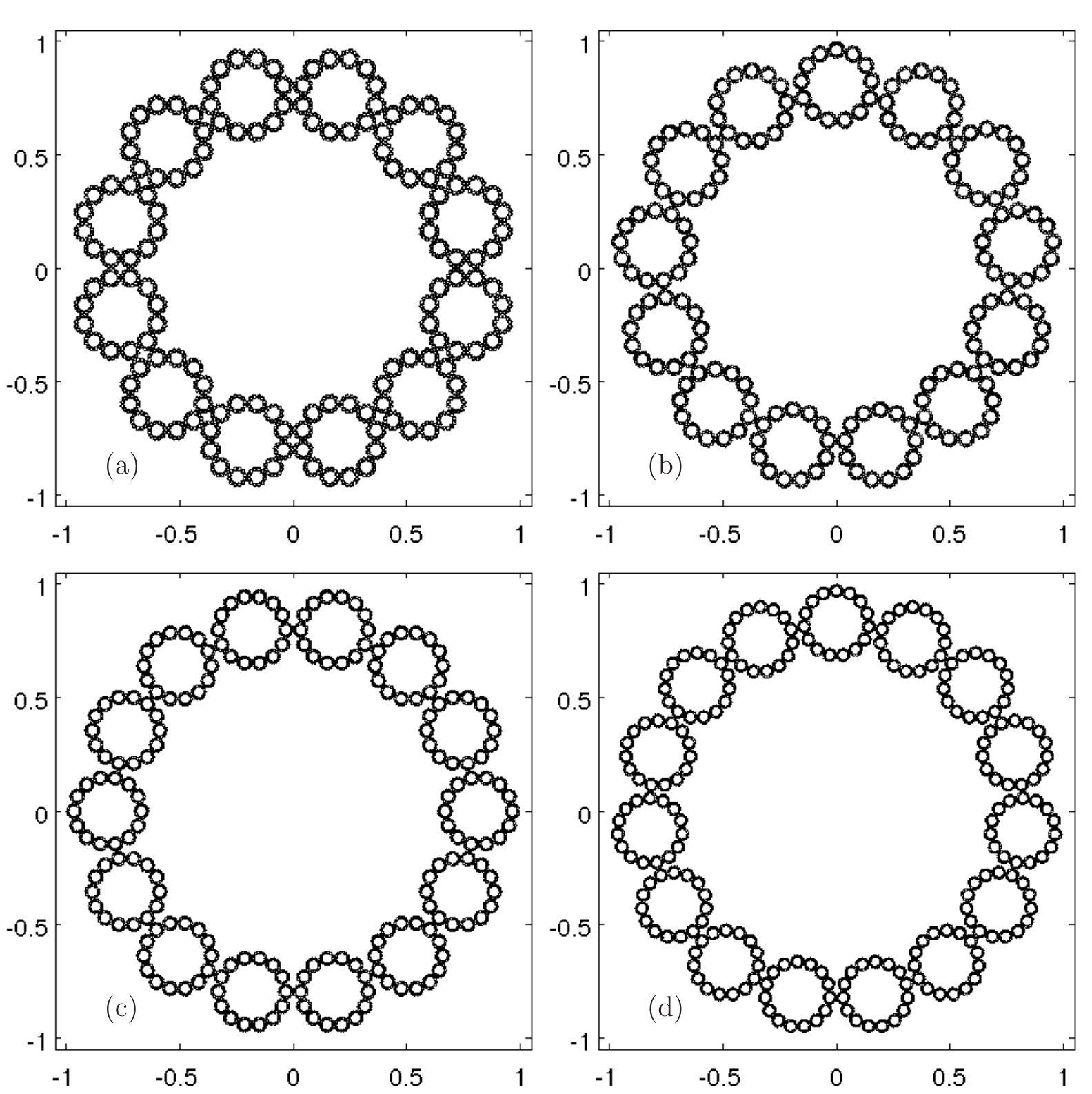}}
}
\end{picture}
\end{center}
\vspace{-0.5cm}
\caption{IFS generated fractal sets made of points that lie in $F_{\{12/4\}}$, $F_{\{13/4\}}$, $F_{\{14/4\}}$ and $F_{\{15/4\}}$ in subpanels (a), (b), (c) and (d), respectively.}\label{fr7}
\end{figure}
\begin{Ex}\label{ex3}
 We will compute the Hausdorff dimensions of the attractors shown in figure \ref{fr7}: 
\begin{itemize}[noitemsep, topsep=0pt]
\item Figure \ref{fr7}(a) $dim_HF_{\{12/4\}}=-ln(12)/ln(P(12,4))\approx1.598670034685813$ 
\item Figure \ref{fr7}(b) $dim_HF_{\{13/4\}}=-ln(13)/ln(P(13,4))\approx1.5653005271788485$ 
\item Figure \ref{fr7}(c) $dim_HF_{\{14/4\}}=-ln(14)/ln(P(14,4))\approx1.5490615012592472$ 
\item Figure \ref{fr7}(d) $dim_HF_{\{15/4\}}=-ln(15)/ln(P(15,4))\approx1.5430579163288531$ 
\end{itemize}
\end{Ex}

\newpage
\begin{figure}[t]
\begin{center}
\begin{picture}(140,80)(0,0)
\put(0,0){
\put(-10,0){\includegraphics{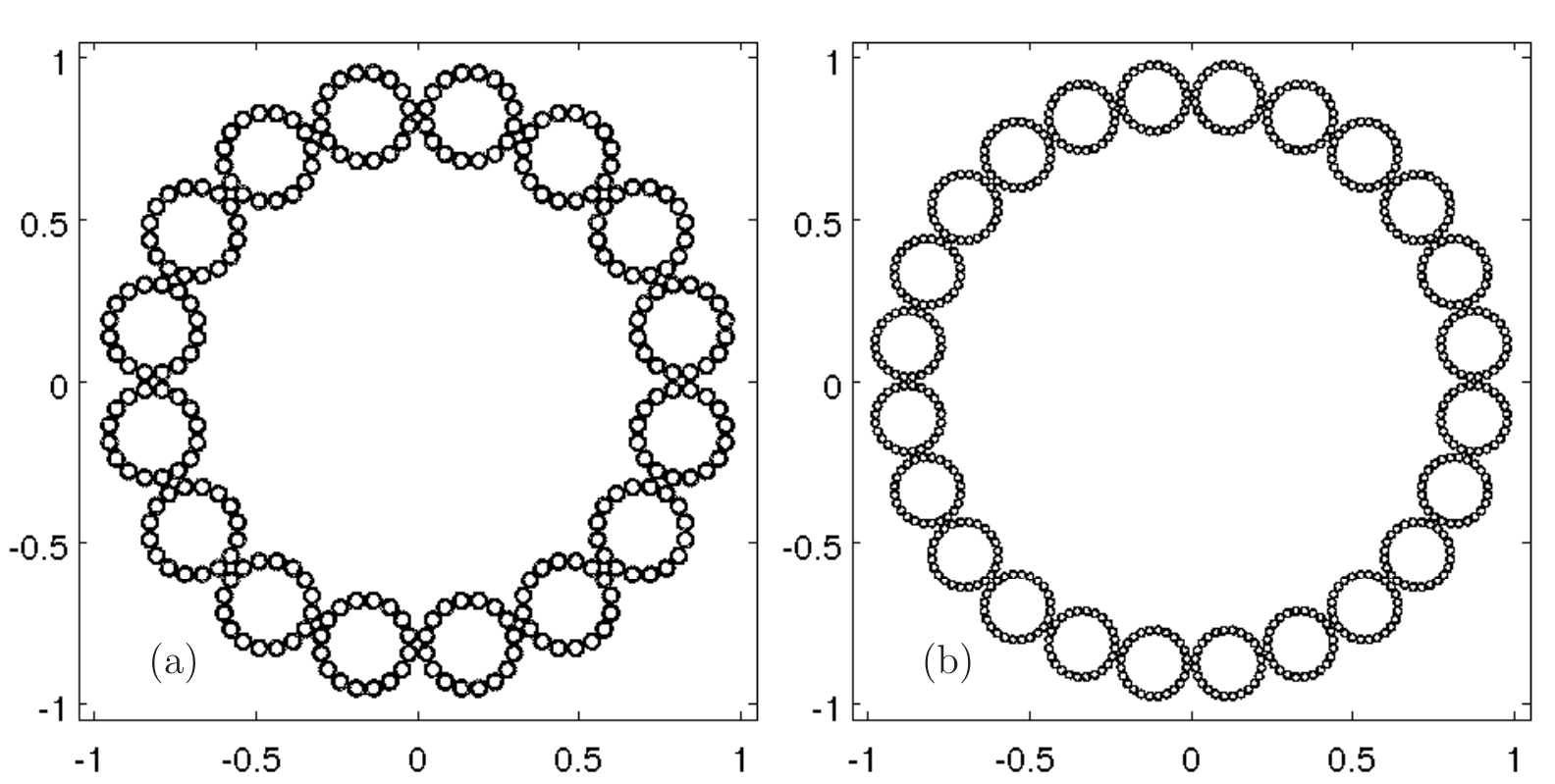}}
}
\end{picture}
\end{center}
\vspace{-0.5cm}
\caption{IFS generated fractal sets made of points that lie in $F_{\{16/5\}}$ and $F_{\{24/7\}}$ in subpanels (a) and (b), respectively.}\label{fr8}
\end{figure}
\begin{Ex}\label{ex4}
 We will compute the Hausdorff dimensions of the attractors shown in figure \ref{fr8}: 
\begin{itemize}[noitemsep, topsep=0pt]
\item Figure \ref{fr8}(a) $dim_HF_{\{16/5\}}=-ln(16)/ln(P(16,5))\approx1.5434949184823248$ 
\item Figure \ref{fr8}(b) $dim_HF_{\{24/7\}}=-ln(24)/ln(P(24,7))\approx1.4772930562556852$ 
\end{itemize}
\end{Ex}

Also, the $dim_HF_{\{n/m\}}$ for $n\in[17,50]$ and $m\in[n/4,n/4+1]$ are as follows:

$\{dim_HF_{\{17/5\}},dim_HF_{\{18/5\}}, ... , dim_HF_{\{50/13\}}\}\approx\{1.5238, 1.5126, 1.5071, 1.5056, 1.4924,\\ 
1.4841, 1.4794, 1.4773, 1.4677, 1.4613, 1.4573, 1.4551, 1.4478, 1.4426, 1.4391, 1.437, 1.4312, 1.4269,\\ 
1.4239, 1.422, 1.4172, 1.4136, 1.4109, 1.4091, 1.4051, 1.402, 1.3997, 1.398, 1.3946, 1.3919, 1.3898,\\ 
1.3883, 1.3853, 1.3829\}$
Finally, for $n=1e+308$, $dim_HF_{\{1e+308/2.5e+307\}}\approx1.001622$. 
\begin{Th}\label{D1}
As $n$ goes to infinity, $dim_HF_{\{n/m\}}$ approaches $1$
\end{Th}
\begin{proof}
Let $s=dim_HF_{\{n/m\}}$ then from
$\displaystyle{\lim_{n \rightarrow \infty} P =\lim_{n \rightarrow \infty} \frac{sin(\pi/n)}{2cos((m-1)\pi/n)sin(m\pi/n)}=}$\\
$\displaystyle{=\lim_{n \rightarrow \infty} \frac{sin(\pi/n)}{2cos(\pi/4)sin(\pi/4)} = \lim_{n \rightarrow \infty} sin(\pi/n) = \lim_{n \rightarrow \infty} \pi/n}$ and $nP^s=1$ we can deduce $s$.\\ 
Thus, $\displaystyle{\lim_{n \rightarrow \infty} s = \lim_{n \rightarrow \infty} \frac{\ln(n)}{\ln(n/\pi)}=\infty/\infty}$, hence
$\displaystyle{\lim_{n \rightarrow \infty} s = \lim_{n \Rightarrow \infty} \frac{\frac{\partial\ln(n)}{\partial n}}{\frac{\partial \ln(n/\pi)}{\partial n}}=1}$
\end{proof}

As $F_{\{n/m\}}$ is inscribed in the same circle in which the initial \{n/m\}-polygon is inscribed, a corollary of Theorem \ref{D1} is that as $n \rightarrow \infty$ the $F_{\{n/m\}}$ is going to be arbitrary close to the circle in which the initial \{n/m\}-polygon is inscribed.

\subsection{Exact drawing of the IFS iterations}

All the figures above were drawn by using a random walk generator that draws points which lie in the IFS attractor \cite{Devaney_1990,Falconer_1990}.
Another way of showing the attractor is by plotting large enough iteration (3th or 4th is usually enough) of the IFS where multiple scaled-down copies of the initial polygon are imaged.
In figure \ref{fr9} an example of this plotting approach is shown where in panel (b) the fourth iteration of the \{5/2\}-polygon looks like figure \ref{fr5}(c) where the same attractor is produced by the random walk technique.
In the next section we will use the latter technique more often for the sake of the clarity of the concepts presented.

\begin{figure}[t]
\begin{center}
\begin{picture}(140,80)(0,0)
\put(0,0){
\put(-10,0){\includegraphics{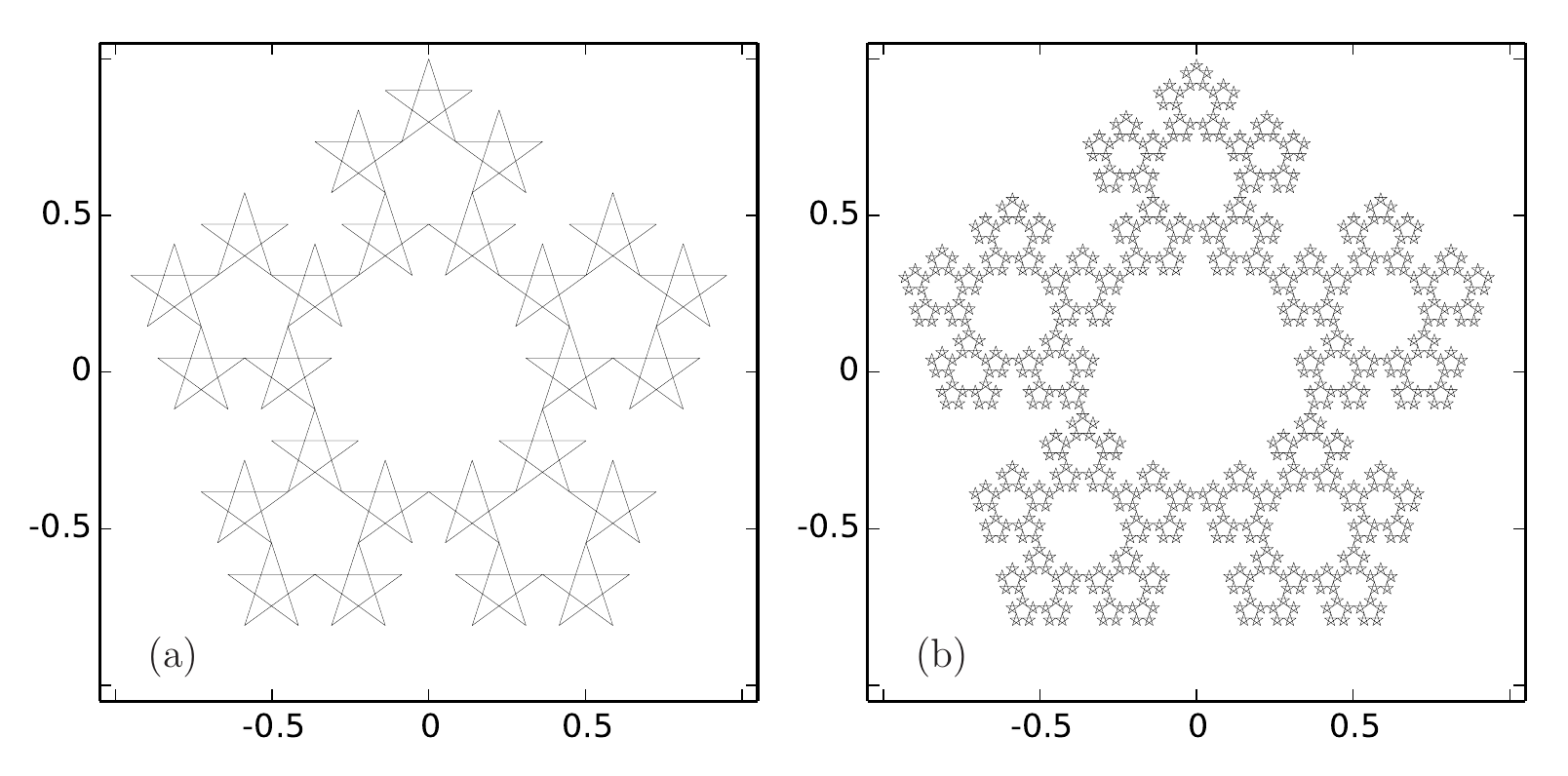}}
}
\end{picture}
\end{center}
\vspace{-0.5cm}
\caption{IFS with initial \{5/2\} star-polygon where the second and the fourth iterations are shown in panels (a) and (b), respectively.}\label{fr9}
\end{figure}

\section{The centre of the circle as an additional point of attraction}\label{sec4}
If we add the centre of the circle as another attracting point that the random generator takes into account, then we can produce non-self-intersecting fractal sets that cover a great amount of the area that is bounded by the unit circle. 
This result is due to the fact that the centre point adds to the IFS attractor (the invariant set) one more scaled copy of the initial star-polygon, hence we need an additional contracting map which we will call $S_c$.
Moreover, if the scaling factor of the central point is carefully computed, one can exploit a number of different features of the star-polygons.
Here we will give a few introductory examples.

 Let us consider an initial polygon \{3/1\}, then $P(3,1)=1/2$ and let us have a central map $S_c$ with the same ratio $1/2$ and rotation $\pi/3$ added to the set of maps $\{S_1,S_2,S_3\}$.
This will result in a triangular shape attractor with a Hausdorf dimension $dim_HF_{\{3/1\}}[L^1,\pi/3]=-ln(4)/ln(P(3,1))=ln(4)/ln(2)=2$.
Thus, the only difference from the attractor shown in figure \ref{fr5}(b) will be the triangular shape.

 We will present the IFS of the $\{5/2\}$ (see figure \ref{fr10}, Example \ref{ex5}) with an attracting centre, where the similarity map corresponding to the centre point has the same scaling factor $P(5,2)$ as the maps that correspond to the vertices of the initial $\{5/2\}$-polygon. 
In figure \ref{fr11} another IFS is shown and its dimension is computed in Example \ref{ex6}.
This fractal has centre-map that not only scales, but also rotates the initial polygon at angle $\pi/5$ while keeping the ratio $P(5,2)$.
Let us also see the IFS of the $\{6/2\}$ with an attracting centre, where the similarity map corresponding to the centre point has a scaling factor $P(6,2)$ and does not imply rotation, presented in figure \ref{fr12} and Example \ref{ex7}.

From the dimensions computed in section \ref{dimensions} we can deduce that as $n$ grows the scaling ratio $P(n,m)$, where $m\in[n/4,n/4+1]$ monotonically decreases and the resulting fractals shrink in dimension. 
Thus, if we want to increase $n$, but keep the attractors with a reasonably high dimension, we can no longer use the same ratio for the centre scaling map as in the cases for $n=5$ and $n=6$.
Therefore, we will define different rules for the scaling ratio of the centre map $S_c$ for any $n$, depending if it is odd or even and if $S_c$ includes any rotation such as $\pi/n$ or it does not.  

\begin{figure}[hbp]
\begin{center}
\begin{picture}(140,160)(0,0)
\put(0,0){
\put(-10,0){\includegraphics{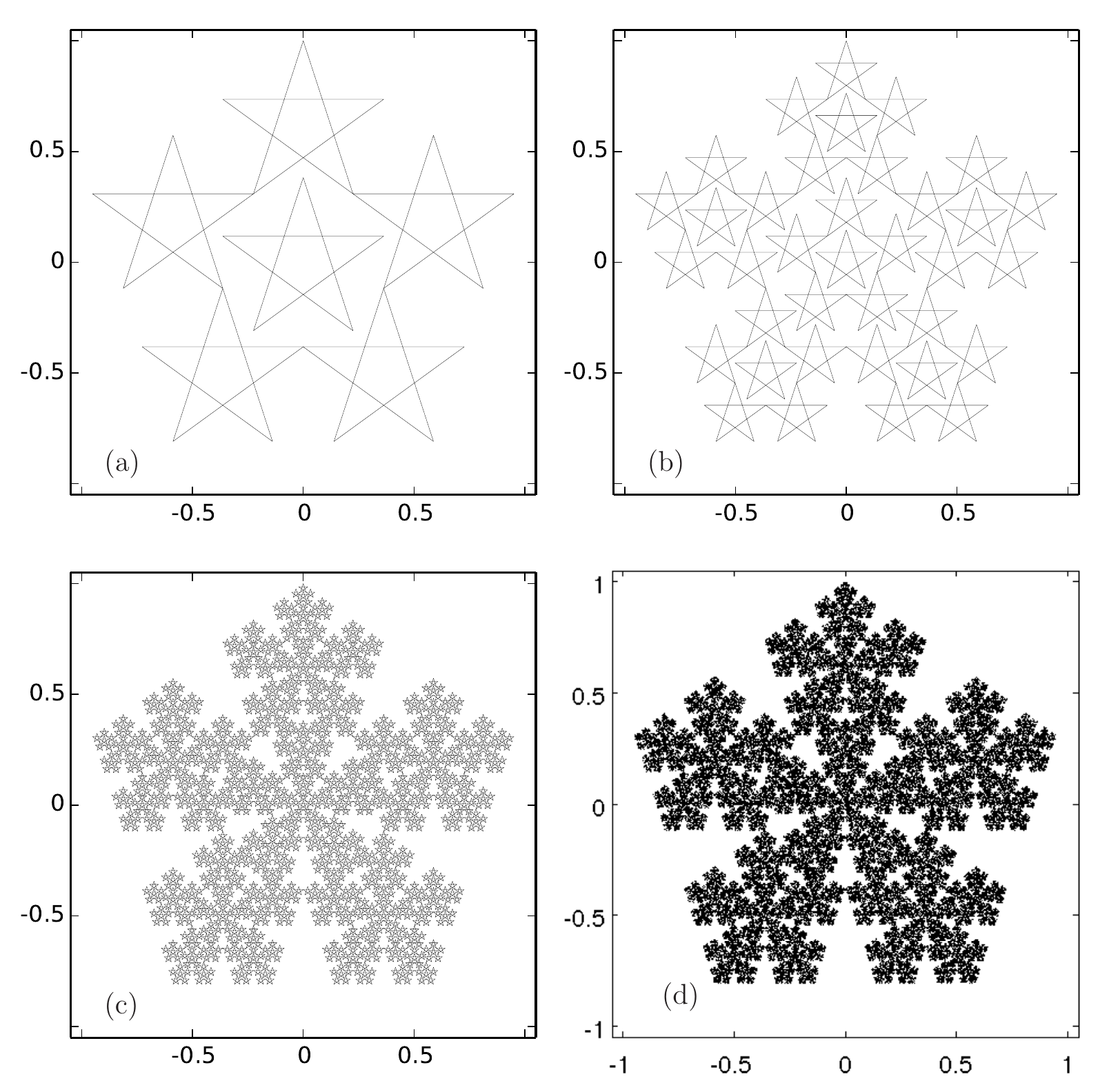}}
}
\end{picture}
\end{center}
\vspace{-0.5cm}
\caption{IFS with initial \{5/2\} star-polygon where the central map $S_c$ has no rotation and uses the ratio $P(5,2)$. 
The first, the second and the fourth iterations are shown in panels (a), (b) and (c) respectively. 
100000 points that lie on the attractor of the IFS are shown in panel (d).}\label{fr10}
\end{figure}
\begin{Ex}\label{ex5}
 We will compute the Hausdorff dimensions of the attractor shown in figure \ref{fr10}(d): 
$dim_HF_{\{5/2\}}[L^1,0]=-ln(6)/ln(P(5,2))\approx1.8617$ 
\end{Ex}
\begin{figure}[hpb]
\begin{center}
\begin{picture}(140,160)(0,0)
\put(0,0){
\put(-10,0){\includegraphics{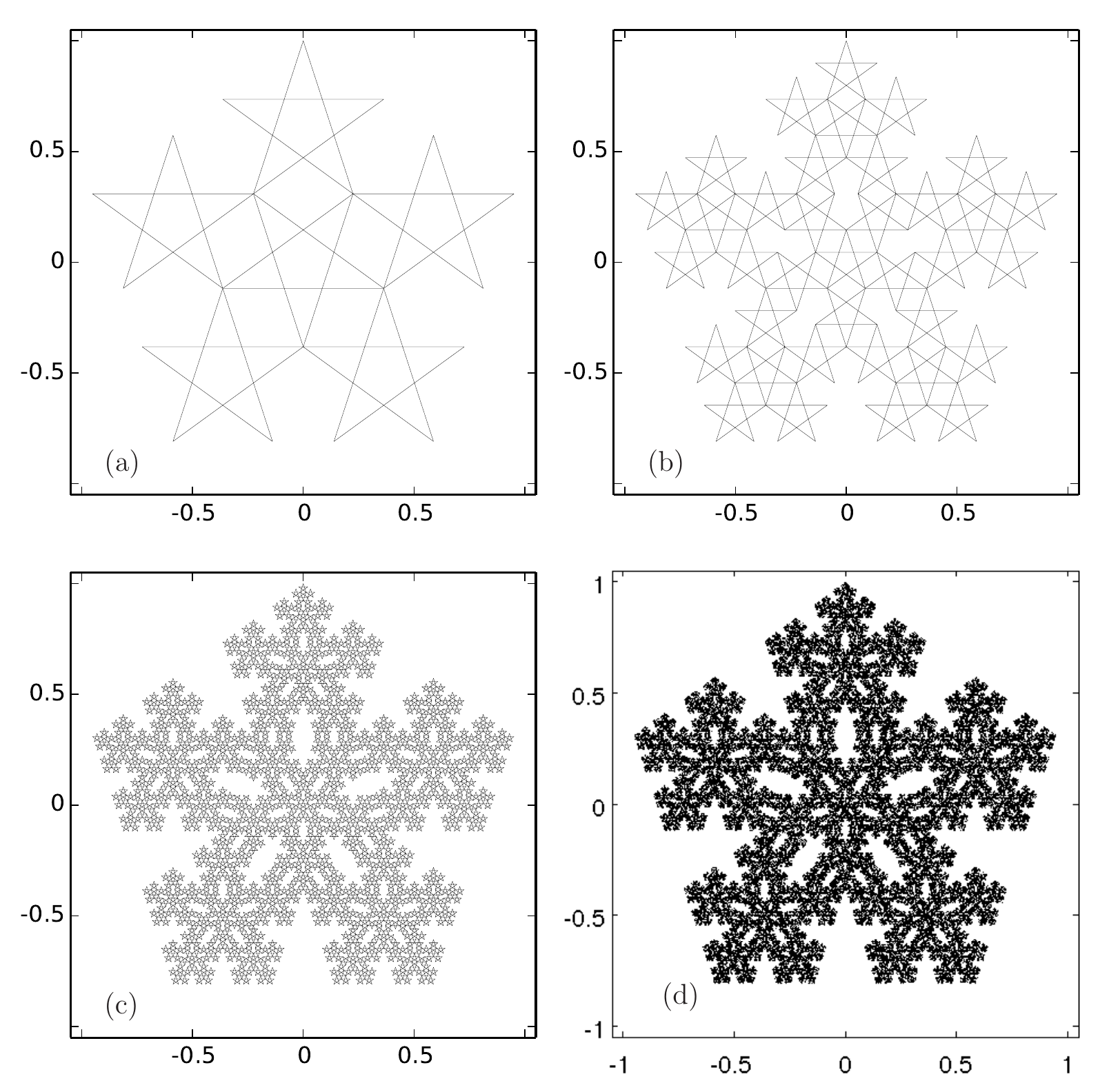}}
}
\end{picture}
\end{center}
\vspace{-0.5cm}
\caption{IFS with initial \{5/2\} star-polygon where the central map $S_c$ has rotation $\pi/5$ and uses the ratio $P(5,2)$. 
The first, the second and the fourth iterations are shown in panels (a), (b) and (c) respectively. 
100000 points that lie on the attractor of the IFS are shown in panel (d).}\label{fr11}
\end{figure}
\begin{Ex}\label{ex6}
 We will compute the Hausdorff dimensions of the attractor shown in figure \ref{fr11}(d): 
$dim_HF_{\{5/2\}}[L^1,\pi/5]=dim_HF_{\{5/2\}}[L^1,0]=-ln(6)/ln(P(5,2))\approx1.8617$  
\end{Ex}
\newpage
\begin{figure}[hbp]
\begin{center}
\begin{picture}(140,160)(0,0)
\put(0,0){
\put(-10,0){\includegraphics{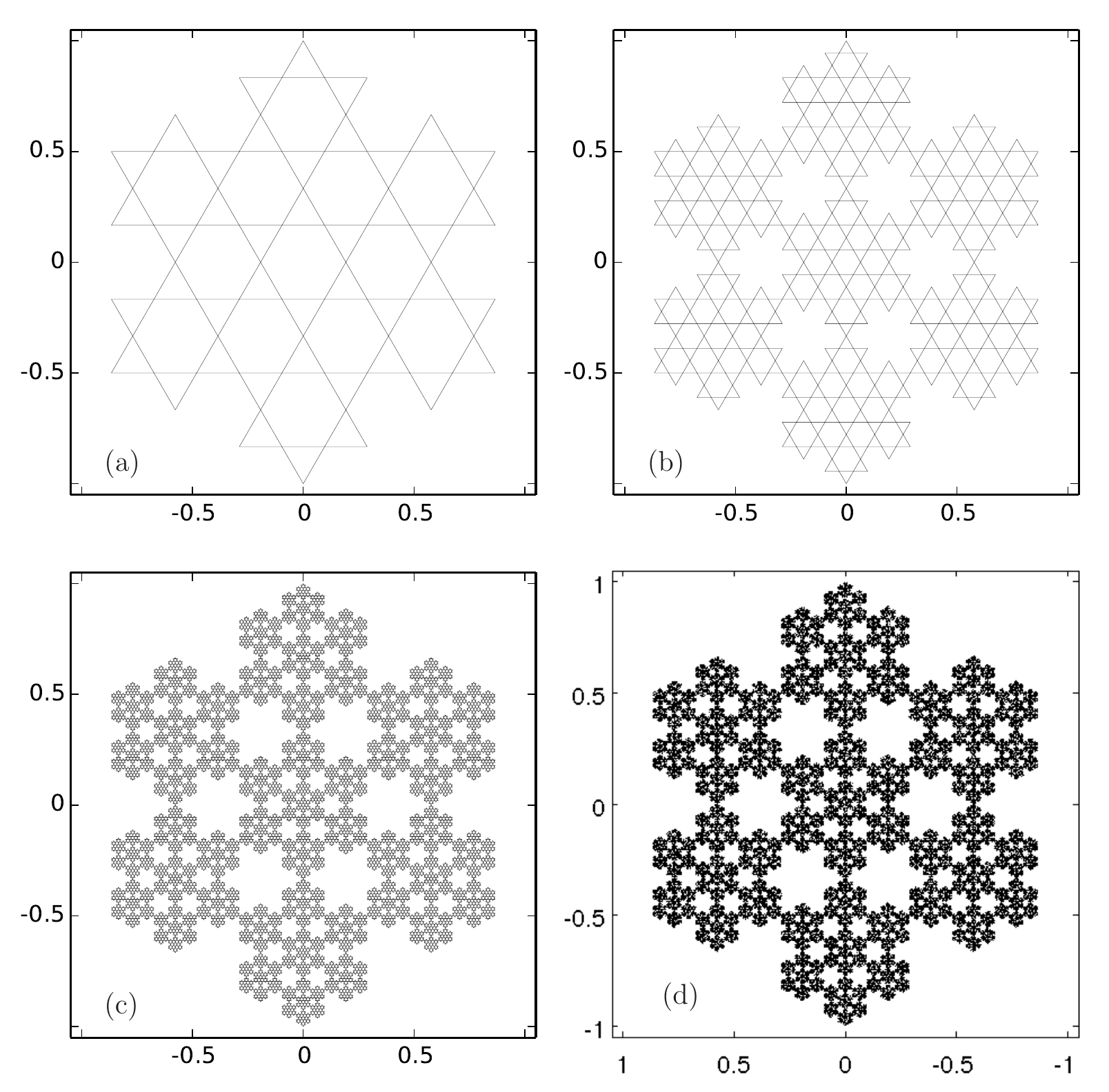}}
}
\end{picture}
\end{center}
\vspace{-0.5cm}
\caption{IFS with initial \{6/2\} star-polygon where the central map $S_c$ has no rotation and uses the ratio $P(6,2)$. 
The first, the second and the fourth iterations are shown in panels (a), (b) and (c) respectively. 
100000 points that lie on the attractor of the IFS are shown in panel (d).}\label{fr12}
\end{figure}
\begin{Ex}\label{ex7}
 We will compute the Hausdorff dimensions of the attractor shown in figure \ref{fr12}(d): 
$dim_HF_{\{6/2\}}[L^0,0]=-ln(7)/ln(P(6,2))\approx1.7712$ 
\end{Ex}
\newpage
\begin{figure}[t]
\begin{center}
\begin{picture}(140,70)(0,0)
\put(0,0){
\put(0,0){\includegraphics{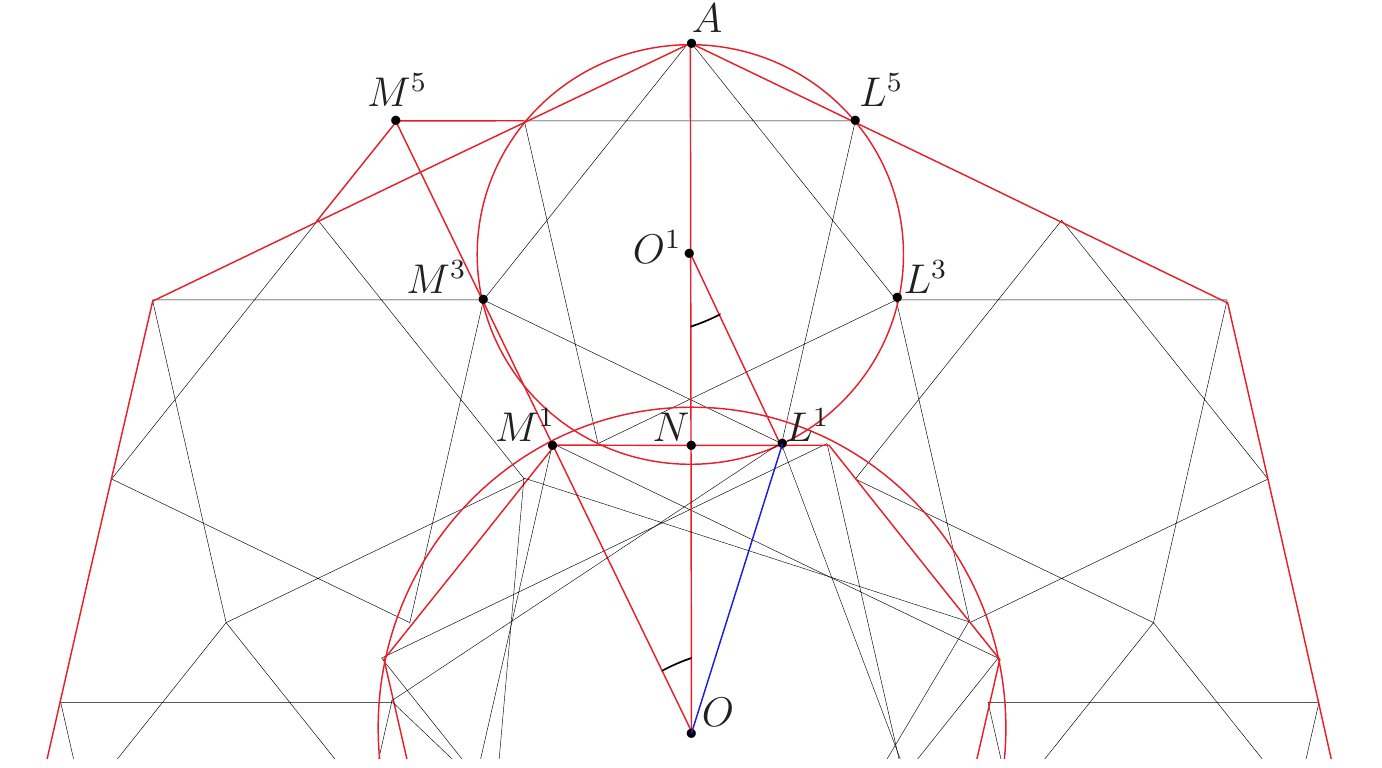}}
}
\end{picture}
\end{center}
\vspace{0.0cm}
\caption{Sketch of a $\{7/2\}$ polygon IFS with a centre map after the first iteration. Shown is the way of computing of the ratio for the centre map.}\label{fr13}
\end{figure}
In order to clearly show how the ratios of $S_c$ are deduced, in figure \ref{fr13} is sketched part of the IFS of $\{7/2\}$-polygon when it is iterated only once. 
There are two centre polygons, one is rotated at angle $\pi/7$ (the one with a vertex at $M^1$) and another rotated at an angle to be computed later (with vertex $L^1$).
Thus, the contraction map that scales the original $\{7/2\}$-polygon towards the point $O$ and images it in one of these two polygons includes rotation as well.
Using the sketch in figure \ref{fr13} we will show that some of the points $M^i$ and $L^i$ will be used as vertices of a central polygon that corresponds to a central map $S_c$, which can be used for non-self-intersecting polygonal IFS.

Let us construct the points $M^1, M^3, M^5, ..., M^{2i+1}$ lying on the line that crosses the line-segment $O^1O$ at angle $\pi/7$.
This line is intersected by the line segments that start from the vertices $L^1, L^3$ and $L^5$ and are perpendicular to the line-segment $O^1O$ resulting in the points $M^1, M^3$ and $M^5$.
Now we will define six different ratios $P_c$ for the centre map $S_c$: $OM^1/OA$, $OM^3/OA$, $OM^5/OA$ and $OL^1/OA$, $OL^3/OA$, $OL^5/OA$.
All of them are defined by the angles $\measuredangle{OO^1L^l}=l\pi/n$ where in general $l=2i+1$ for $i\geq0, i\in\mathbb{Z}$ when the polygons are odd-sided and $l=2i$ when the polygons are even-sided.
For $l=1,3,5$, $ON=OA-O^1A-O^1L^l\cos(l\pi/7)$ and $OM^l=ON/\cos(\pi/7)$, thus $OM^l/OA=1/\cos(\pi/7)-(O^1A/OA)(1/\cos(\pi/7)-\cos(l\pi/7)/\cos(\pi/7))$. 
This equation also holds for $l=7$, where $A\equiv L^7$ and if apply equation (\ref{eq1}), so that $O^1A/OA=P(7,2)$, then $OM^l/OA=1/\cos(\pi/7)-P(7,2)/\cos(\pi/7)-P(7,2)\cos(l\pi/7)/\cos(\pi/7)$.
Also, if generalised for an arbitrary initial $\{n/m\}$-polygon, it leads to the following equation: 

\begin{align}\label{eq5}
&\displaystyle{\frac{OM^l}{OA}}=\displaystyle{\frac{1-P(n,m)-P(n,m)\cos(l\pi/n)}{\cos(\pi/n)}} \\
&\measuredangle{O^1OM^l}=\pi/n \nonumber \\
&0\leq l\leq n, l=2i \text{ if $n$-even}, l=2i+1 \text{ if $n$-odd}, i\geq0, i\in\mathbb{Z} \nonumber
\end{align}

Another ratio that may be used for the map $S_c$ is $OL^l/OA$.
Here $NL^1=O^1L^1\sin(\pi/7)$, hence, $\tan(\measuredangle{O^1OL^1})=\displaystyle{\frac{O^1L^1\sin(\pi/7)}{OA-O^1A-O^1L^1\cos(\pi/7)}}$ which for an arbitrary $l$ will become\\ $\tan(\measuredangle{O^1OL^l})=\displaystyle{\frac{O^1L^l\sin(l\pi/7)}{OA-O^1A-O^1L^l\cos(l\pi/7)}}$. 
Now let us take into account that $O^1L^l=O^1A$ and $O^1A/OA=P(7,2)$, hence $\tan(\measuredangle{O^1OL^l})=\displaystyle{\frac{P(7,2)\sin(l\pi/7)}{1-P(7,2)-P(7,2)\cos(l\pi/7)}}$. 
Finally, for an arbitrary initial $\{n/m\}$-polygon we can deduce the angle $\measuredangle{O^1OL^l}$ and from $(OL^l)^2=(NL^l)^2+(ON)^2$ we can also deduce the ratio $OL^l/OA$ as follows:
\begin{align}\label{eq6}
&\displaystyle{\frac{OL^l}{OA}}=\sqrt{2P(n,m)(P(n,m)-1)(1+\cos(l\pi/n))+1} \\
&\gamma(n,m,l)=\measuredangle{O^1OL^l}=\arctan\left(\displaystyle{\frac{P(n,m)\sin(l\pi/n)}{1-P(n,m)-P(n,m)\cos(l\pi/n)}}\right) \nonumber \\
&0\leq l\leq n, l=2i \text{ if $n$-even}, l=2i+1 \text{ if $n$-odd}, i\geq0, i\in\mathbb{Z} \nonumber
\end{align}
Now we can look back at figures \ref{fr10}, \ref{fr11} and \ref{fr12} and understand how they are constructed.
In example \ref{ex5}, figure \ref{fr10}, the contraction ratio of $S_c$ is $OL^1/OA$, and the attractor of the IFS is denoted as $F_{\{5/2\}}[L^1,0]$, where 0 indicates the angle of rotation that $S_c$ has.
In this case of $\{5/2\}$-polygon $OM^1/OA=OL^1/OA$, so it does not matter if $L^1$ or $M^1$ is used for the notation.
In the other examples where $OL^l/OA=OM^l/OA$, again $L^l$ will be used as notation.
In example \ref{ex6}, figure \ref{fr11} the contraction ratio of $S_c$ is again $OL^1/OA$, but here we have rotation at angle $\pi/5$, thus the attractor of the IFS is denoted as $F_{\{5/2\}}[L^1,\pi/5]$.
And finally in example \ref{ex7}, figure \ref{fr12} the contraction ratio of $S_c$ is $OL^0/OA$ and the attractor of the IFS is denoted $F_{\{6/2\}}[L^0,0]$. 

\subsection{Even $n$}

In this subsection we will take a close look at the IFS that originates from even sided star-polygons.
Firstly, we should note that the same way as $F_{\{6/2\}}[L^0,0]$, for any even $n$, $F_{\{2i/m\}}[L^0,0]$ will always be a non-self-intersecting attractor if $m\in[n/4,n/4+1]$ and $i\in\mathbb{N}$. 
Now, the first example has \{6/2\} as initial polygon, and $S_c$ has scaling ratio $OL^2/OA$ and rotation $\pi/6$.
The resulting fractal can be seen in figure \ref{fr14}, where from the random generated attractor, see panel (d), we can expect the exact dimension of 2.
Indeed, this is analytically proven in the computations of example \ref{ex8}.
Similarly, the constructions and the attractors of $F_{\{8/2\}}[L^2,\frac{\pi}{8}]$ and $F_{\{8/3\}}[L^2,\frac{\pi}{8}]$ are shown in figures \ref{fr15} and \ref{fr16}.
They have equal dimension computed in example \ref{ex9}.
Another pair of attractors that have central map and originate from \{8/2\}-star polygon are the $F_{\{8/2\}}[L^0,0]$ and $F_{\{8/3\}}[L^0,0]$ shown in figures \ref{fr17} and \ref{fr18}.
They have equal dimension computed in example \ref{ex10}.
The last four examples of attractors clearly show that the scaling ratio and the number of the vertices are the parameters that define the attractor of the IFS.  
 
\newpage
\begin{figure}[hbp]
\begin{center}
\begin{picture}(140,160)(0,0)
\put(0,0){ 
\put(-10,0){\includegraphics{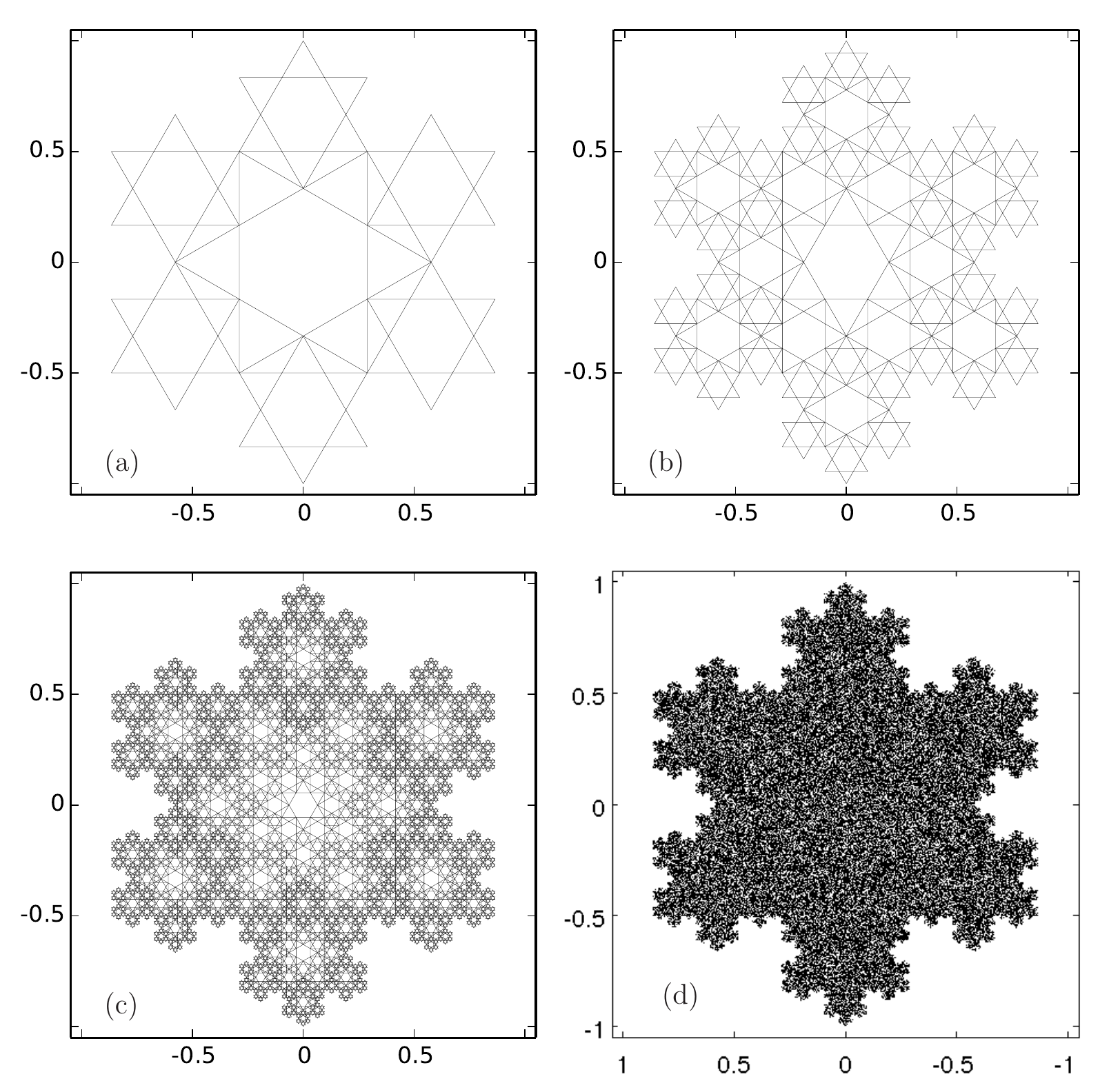}}
}
\end{picture}
\end{center}
\vspace{-0.5cm}
\caption{IFS with initial \{6/2\} star-polygon where $S_c$ has scaling ratio $OL^2/OA$ and rotation $\pi/6$. 
The first, the second and the fourth iterations are shown in panels (a), (b) and (c) respectively. 
100000 points that lie on the attractor of the IFS are shown in panel (d).}\label{fr14}
\end{figure}
\begin{Ex}\label{ex8}
 We will compute the Hausdorff dimensions of the attractor shown in figure \ref{fr14}(d): 
$6P(6,2)^{dim_HF_{\{6/2\}}[L^2,\frac{\pi}{6}]}+\sqrt{2P(6,2)(P(6,2)-1)(1+\cos(2\pi/6))+1}^{dim_HF_{\{6/2\}}[L^2,\frac{\pi}{6}]}=1$ \\
$6(1/3)^{dim_HF_{\{6/2\}}[L^2,\frac{\pi}{6}]}+\sqrt{(1/3)}^{dim_HF_{\{6/2\}}[L^2,\frac{\pi}{6}]}=1$. Let $y=\sqrt{(1/3)}^{dim_HF_{\{6/2\}}[L^2,\frac{\pi}{6}]}$\\
Hence, $6y^2+y-1=0$ and $y_{1,2}=1/3;-1/2$, therefore as $y\geq0$\\
$\sqrt{(1/3)}^{dim_HF_{\{6/2\}}[L^2,\frac{\pi}{6}]}=1/3\rightarrow dim_HF_{\{6/2\}}[L^2,\frac{\pi}{6}]=2$.
\end{Ex}

\newpage
\begin{figure}[hbp]
\begin{center}
\begin{picture}(140,160)(0,0)
\put(0,0){ 
\put(-10,0){\includegraphics{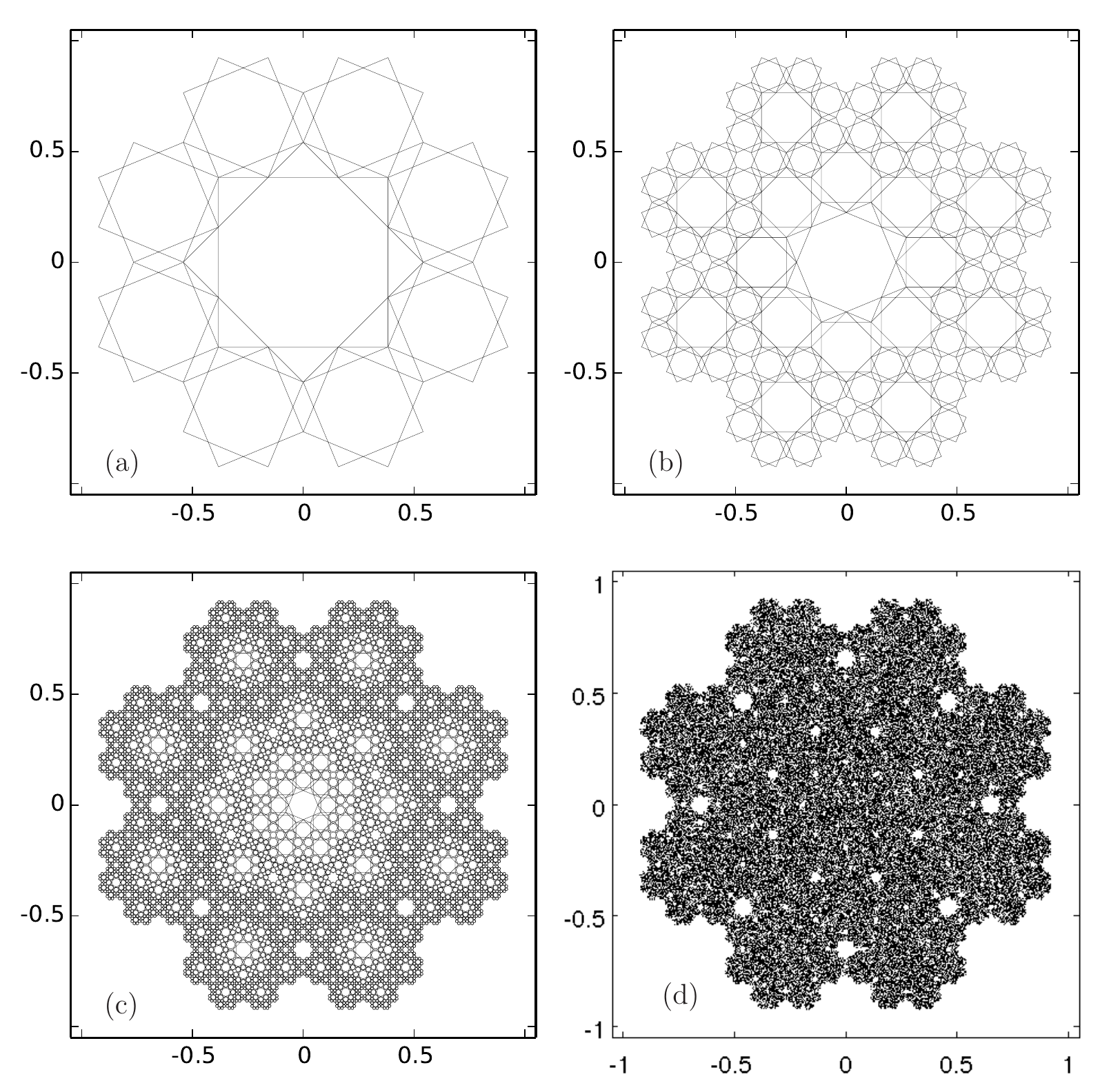}}
}
\end{picture}
\end{center}
\vspace{-0.5cm}
\caption{IFS with initial \{8/2\} star-polygon where $S_c$ has scaling ratio $OL^2/OA$ and rotation $\pi/8$. 
The first, the second and the fourth iterations are shown in panels (a), (b) and (c) respectively. 
100000 points that lie on the attractor of the IFS are shown in panel (d).}\label{fr15}
\end{figure}
\begin{Ex}\label{ex9}
 The Hausdorff dimensions of the attractor shown in figure \ref{fr15}(d) is\\ 
$dim_HF_{\{8/2\}}[L^2,\frac{\pi}{8}]\approx1.9799$
\end{Ex}

\newpage
\begin{figure}[hbp]
\begin{center}
\begin{picture}(140,160)(0,0)
\put(0,0){ 
\put(-10,0){\includegraphics{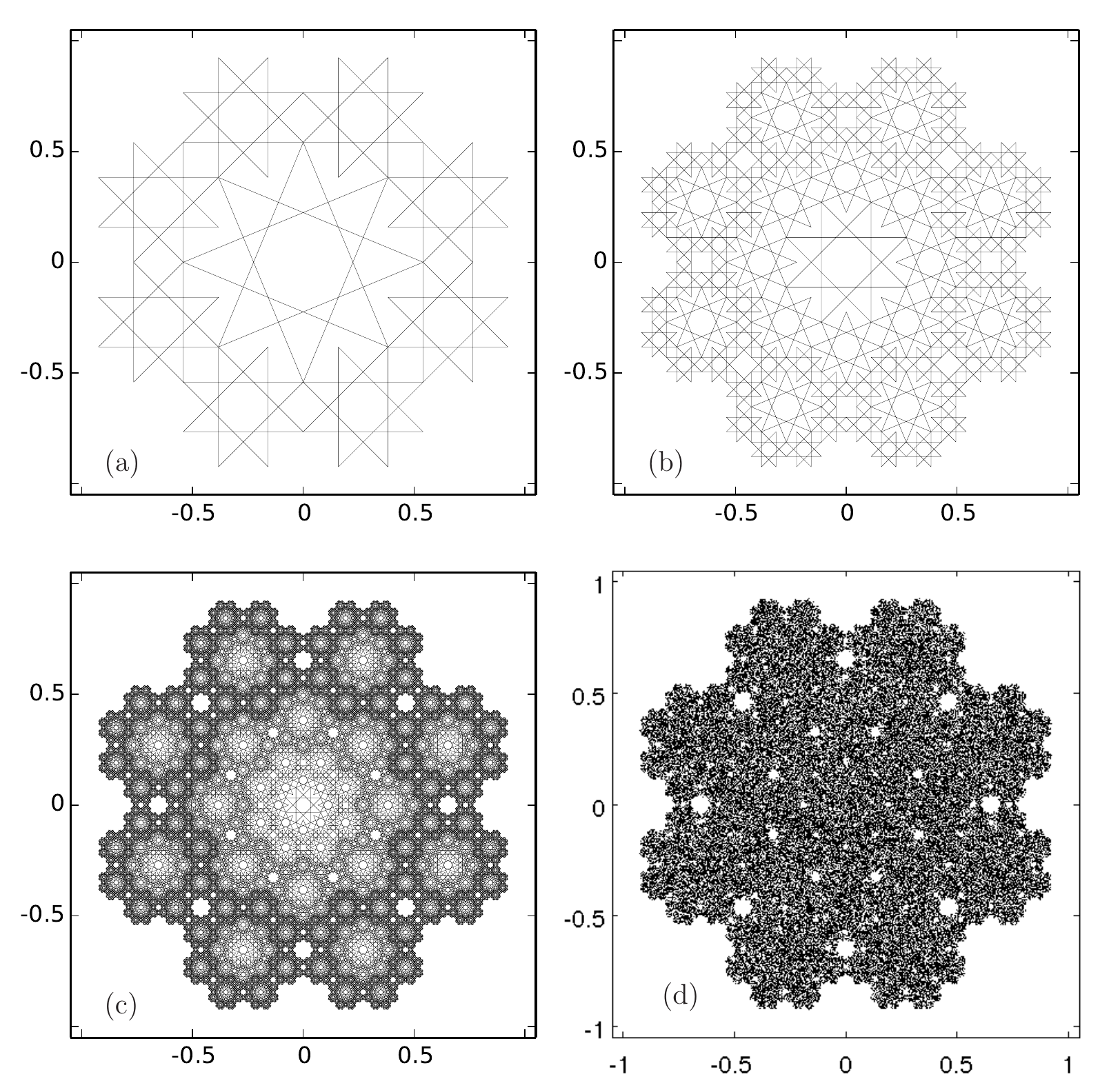}}
}
\end{picture}
\end{center}
\vspace{-0.5cm}
\caption{IFS with initial \{8/3\} star-polygon where $S_c$ has scaling ratio $OL^2/OA$ and rotation $\pi/8$. 
The first, the second and the fourth iterations are shown in panels (a), (b) and (c) respectively. 
100000 points that lie on the attractor of the IFS are shown in panel (d).}\label{fr16}
\end{figure}
\newpage
\begin{figure}[hbp]
\begin{center}
\begin{picture}(140,160)(0,0)
\put(0,0){ 
\put(-10,0){\includegraphics{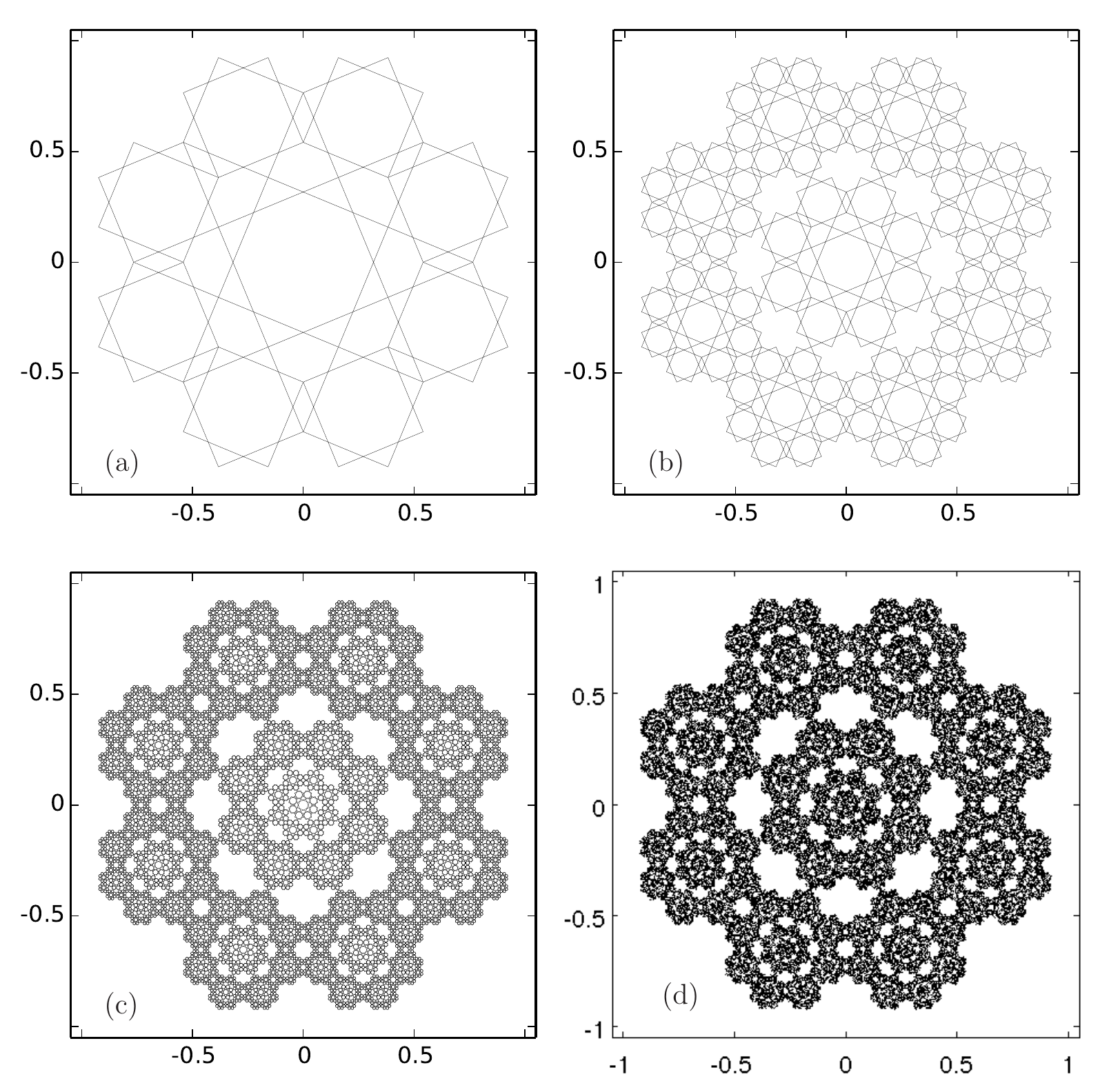}}
}
\end{picture}
\end{center}
\vspace{-0.5cm}
\caption{IFS with initial \{8/2\} star-polygon where $S_c$ has scaling ratio $OL^0/OA$ and no rotation. 
The first, the second and the fourth iterations are shown in panels (a), (b) and (c) respectively. 
100000 points that lie on the attractor of the IFS are shown in panel (d).}\label{fr17}
\end{figure}
\begin{Ex}\label{ex10}
 The Hausdorff dimensions of the attractor shown in figure \ref{fr17}(d) is\\ 
$dim_HF_{\{8/2\}}[L^0,0]\approx1.8678$
\end{Ex}

\newpage
\begin{figure}[hbp]
\begin{center}
\begin{picture}(140,160)(0,0)
\put(0,0){ 
\put(-10,0){\includegraphics{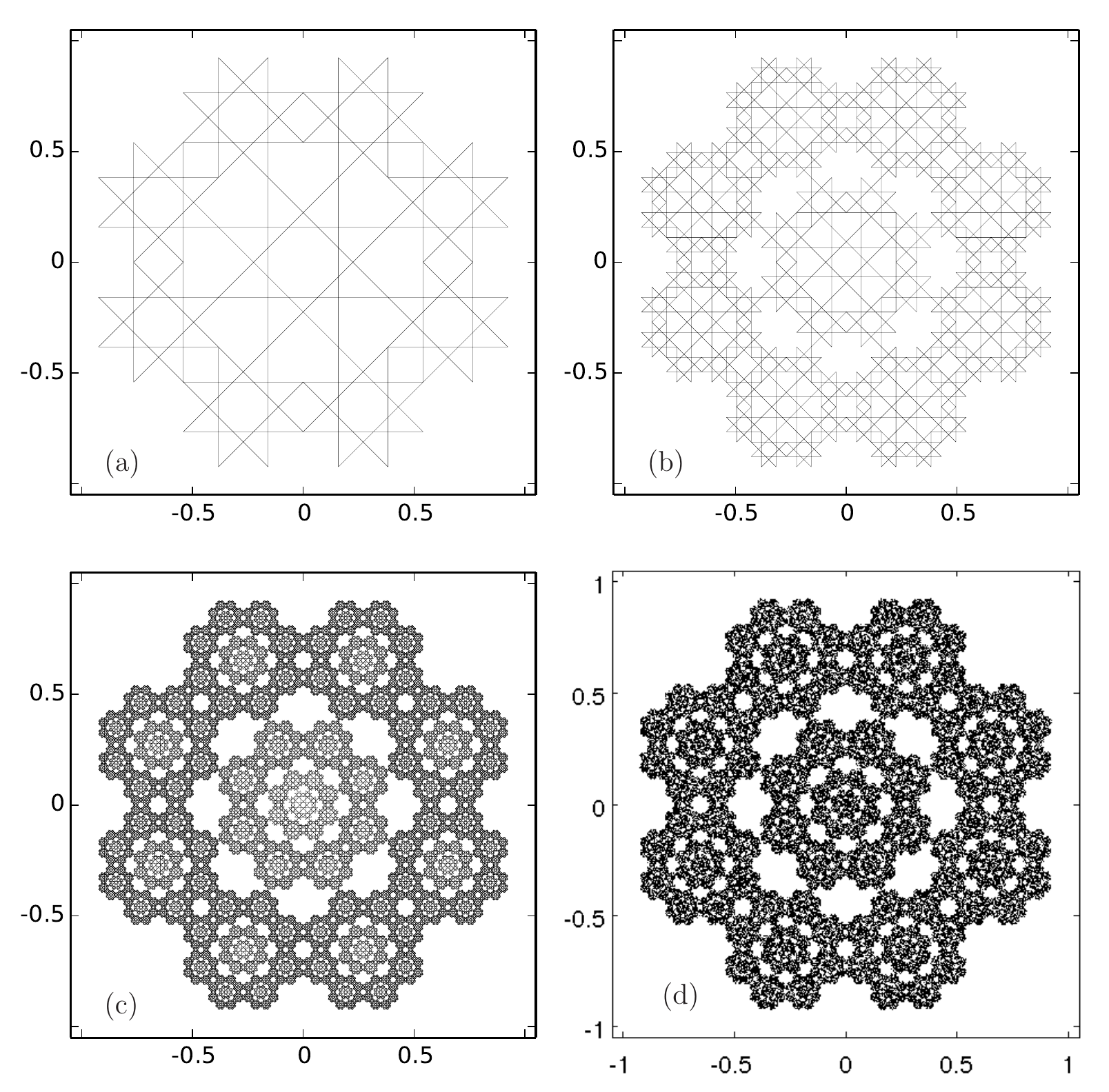}}
}
\end{picture}
\end{center}
\vspace{-0.5cm}
\caption{IFS with initial \{8/3\} star-polygon where $S_c$ has scaling ratio $OL^0/OA$ and no rotation. 
The first, the second and the fourth iterations are shown in panels (a), (b) and (c) respectively. 
100000 points that lie on the attractor of the IFS are shown in panel (d).}\label{fr18}
\end{figure}

\subsection{Odd $n$}
In this subsection we will take a close look at the IFS that originates from odd-sided star-polygons.
The first example has \{7/2\} as an initial polygon and $S_c$ has a scaling ratio $OM^1/OA$ and rotation $\pi/7$; see Eqs. (\ref{eq5}).
The resulting fractal can be seen in figure \ref{fr19}, where the first, the second and the fourth iterations are in panels (a), (b) and (c), while the randomly generated attractor is in panel (d).
The Hausdorf dimension of the attractor $F_{\{7/2\}}[M^1,\pi/7]$ is computed in example \ref{ex11} to be 1.8773.
Another fractal that originates from a \{7/2\}-polygon is shown in figure \ref{fr20}.
Here the scaling ratio is $OL^1/OA$ and the angle of rotation is computed using equations (\ref{eq6}), which leads to polygons that meet at their vertices.
The dimension of $F_{\{7/2\}}[L^1,\gamma(7,2,1)]$ is computed in example \ref{ex12} to be 1.8564.

The technique that uses the ratio $OL^1/OA$ and the angle from equations (\ref{eq6}) can also be used for producing non-intersecting self-similar fractals for any odd $n$. 
Therefore, for any $n$ we can generate a $n$-flake which will be either $F_{\{n/m\}}[L^0,0]$ if $n$ is even or $F_{\{n/m\}}[L^1,\gamma(n,m,1)]$ if $n$ is odd.

\subsection{The dimension of $\infty$-flake}

For any $n$, $m\in[n/4,n/4+1]$ and $i\in\mathbb{N}$, by construction if $n=2i$ then $F_{\{n/m\}}[L^0,0]$ is non-self-intersecting, and by construction if $n=2i+1$ then $F_{\{n/m\}}[L^1,\gamma(n,m,1)]$ is non-self-intersecting.
Therefore, equations (\ref{eq1}), (\ref{eq4}) and (\ref{eq6}) can be used for the corresponding $dim_HF_{\{n/m\}}[L^0,0]$ and $dim_HF_{\{n/m\}}[L^1,\gamma(n,m,1)]$ to be obtained.
\begin{Th}\label{D2} 
As $n$ goes to infinity, $dim_HF_{\{n/m\}[L^0,0]}$ and $dim_HF_{\{n/m\}[L^1,\gamma(n,m,1)]}$ approach $2$.
\end{Th}
\begin{proof}
Both dimensions can be deduced from the equation\\
\begin{equation*} 
nP(n,m)^s+\sqrt{2P(n,m)(P(n,m)-1)(1+\cos(l\pi/n))+1}^s=1,
\end{equation*}
where $s$ denotes $dim_HF_{\{n/m\}[L^0,0]}$ or $dim_HF_{\{n/m\}[L^1,\gamma(n,m,1)]}$.
The latter equation can be modified to 
\begin{equation*}
\displaystyle{P^s=\frac{1}{n}-\frac{\sqrt{2P(n,m)(P(n,m)-1)(1+\cos(l\pi/n))+1}^s}{n}},
\end{equation*}
from where:
\begin{align*}
s&=\lim_{n\rightarrow\infty}\frac{\ln\Big(\displaystyle\frac{1}{n}-\frac{\sqrt{2P(n,m)(P(n,m)-1)(1+\cos(l\pi/n))+1}^s}{n}\Big)}{\ln(P)}=\\
&=\lim_{n\rightarrow\infty} \frac{\ln\Big(\displaystyle\frac{1}{n}-\frac{\sqrt{4P(n,m)(P(n,m)-1)+1}^s}{n}\Big)}{\ln(P)}=\\
&=\lim_{n\rightarrow\infty} \frac{\ln\Big(\displaystyle\frac{1}{n}-\frac{(1-2P)^s}{n}\Big)}{\ln(P)}=\lim_{n\rightarrow\infty} \frac{\ln\Big(\displaystyle\frac{1}{n}-\frac{(1-2\pi/n)^s}{n}\Big)}{\ln(\pi/n)}
\end{align*}
Let us substitute $\nu=\pi/n$, hence
\begin{align*}
s&=\lim_{\nu\rightarrow 0} \frac{\ln\Big(\displaystyle\frac{\nu}{\pi}(1-(1-2\nu)^s)\Big)}{\ln(\nu)}= \lim_{\nu\rightarrow 0} \frac{\ln\Big(\displaystyle\frac{\nu}{\pi}\Big)}{\ln(\nu)}+\lim_{\nu\rightarrow 0} \frac{\ln(1-(1-2\nu)^s)}{\ln(\nu)}=\\
&=1+\frac{-\infty}{-\infty}=1+\lim_{\nu\rightarrow 0} \frac{2s\nu(1-2\nu)^{s-1}}{1-(1-2\nu)^s}=1+\lim_{\nu\rightarrow 0} \frac{2s\nu(1-(s-1)2\nu+O(2))}{1-(1-s2\nu+O(2))}=\\
&=1+\lim_{\nu\rightarrow 0} \frac{2s\nu+O(2)}{2s\nu+O(2)}=2
\end{align*}
\end{proof}
\newpage
\begin{figure}[hbp]
\begin{center}
\begin{picture}(140,160)(0,0)
\put(0,0){ 
\put(-10,0){\includegraphics{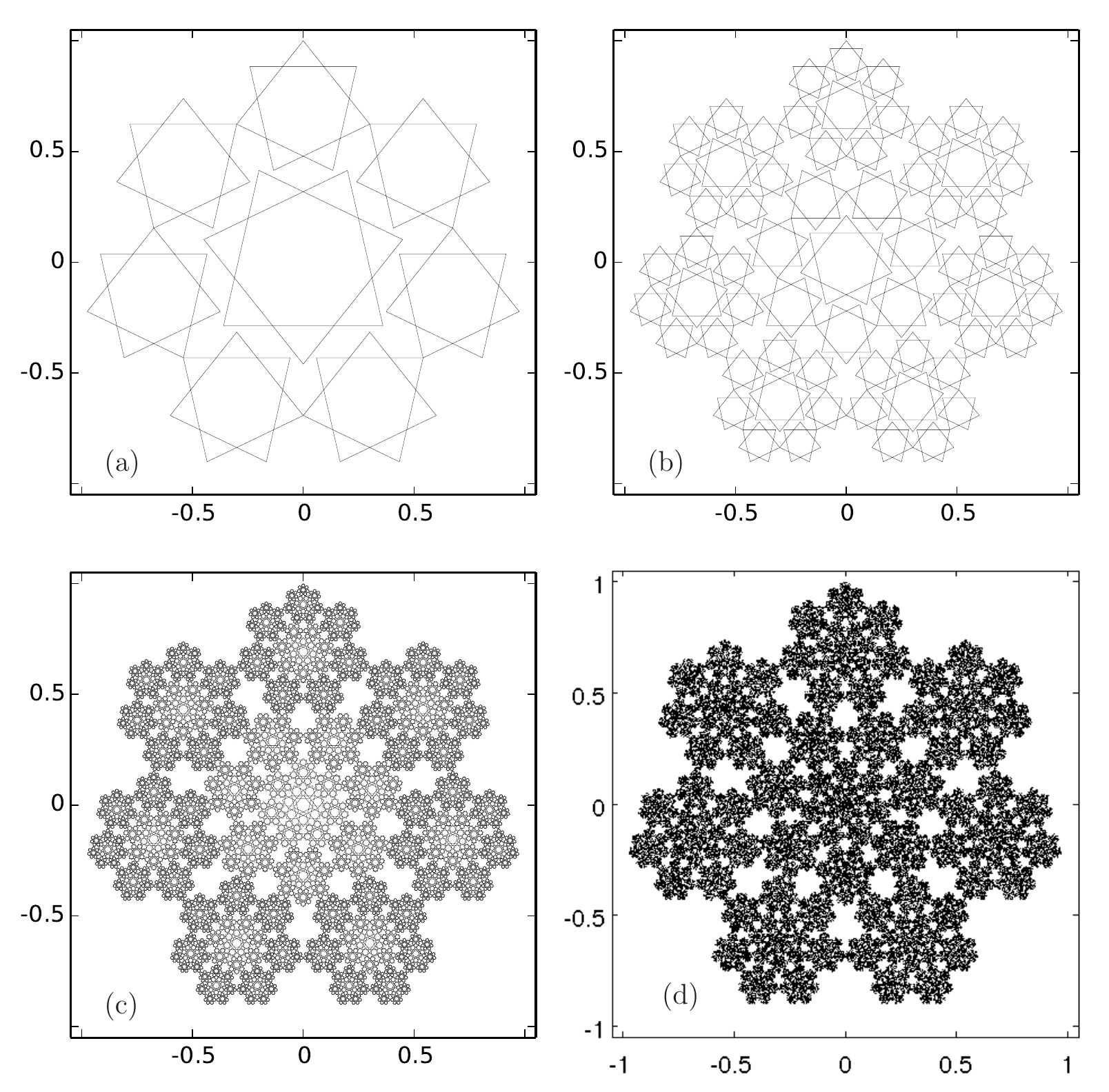}}
}
\end{picture}
\end{center}
\vspace{-0.5cm}
\caption{IFS with initial \{7/2\} star-polygon where $S_c$ has scaling ratio $OM^1/OA$ and rotation $\pi/7$. 
The first, the second and the fourth iterations are shown in panels (a), (b) and (c) respectively. 
100000 points that lie on the attractor of the IFS are shown in panel (d).}\label{fr19}
\end{figure}
\begin{Ex}\label{ex11}
 The Hausdorff dimensions of the attractor shown in figure \ref{fr19}(d) is\\ 
$dim_HF_{\{7/2\}}[M^1,\frac{\pi}{7}]\approx1.8773$ 
\end{Ex}
\newpage
\begin{figure}[hbp]
\begin{center}
\begin{picture}(140,160)(0,0)
\put(0,0){ 
\put(-10,0){\includegraphics{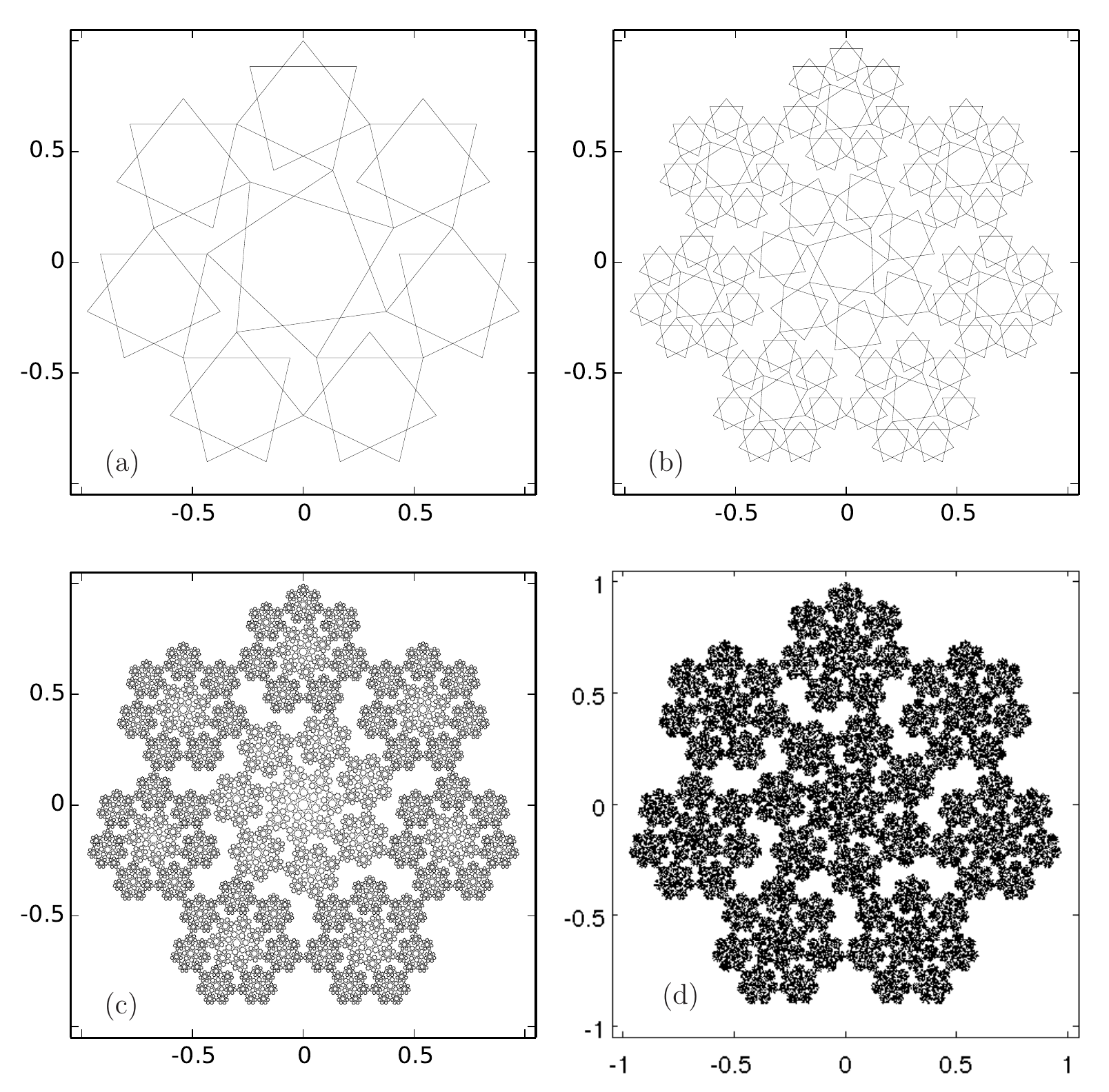}}
}
\end{picture}
\end{center}
\vspace{-0.5cm}
\caption{IFS with initial \{7/2\} star-polygon where $S_c$ has scaling ratio $OL^1/OA$ and rotation $\gamma(7,2,1)$ radians. 
The first, the second and the fourth iterations are shown in panels (a), (b) and (c) respectively. 
100000 points that lie on the attractor of the IFS are shown in panel (d).}\label{fr20}
\end{figure}
\begin{Ex}\label{ex12}
 The Hausdorff dimensions of the attractor shown in figure \ref{fr20}(d) is\\ 
$dim_HF_{\{7/2\}[L^1,\gamma(7,2,1)]}\approx1.8564$
\end{Ex}

\section{Special cases and equivalent IFS attractors}\label{sec5}
\begin{figure}[hbp]
\begin{center}
\begin{picture}(140,160)(0,0)
\put(0,0){
\put(-10,0){\includegraphics{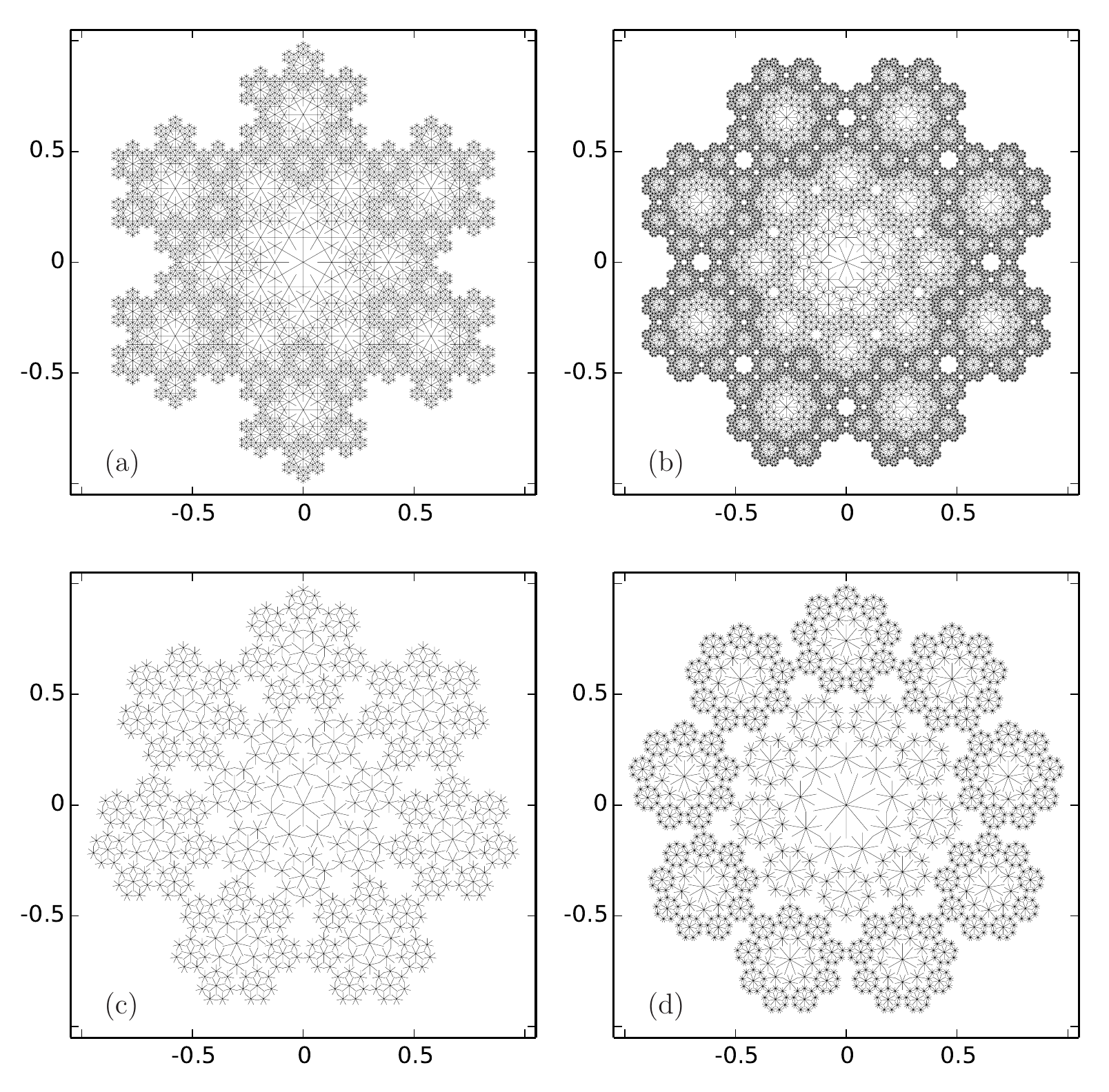}}
}
\end{picture}
\end{center}
\vspace{-0.5cm}
\caption{ In subpanels (a) and (b) we can see the third iteration of the IFS that originate from \{6/3\} and \{8/4\}, respectively. 
Both $S_c$ have scaling ratio $OL^2/OA$, where in panel (a) the rotation is $\pi/6$ and in (b) it is $\pi/8$.
In subpanels (c) and (d) we can see the third iteration of the IFS that originate from \{7/3.5\} and \{9/4.5\}, respectively. 
Both $S_c$ have scaling ratio $OM^1/OA$, where in panel (c) the rotation is $\pi/7$ and in (d) it is $\pi/9$.}\label{fr21}
\end{figure}
As in the previous sections, here we assume $n\geq2$, $n\in\mathbb{Z}$, $m\in\mathbb{Z}$, $0\leq m<n$, and $i\in\mathbb{Z}$, $i\geq 0$.
Let us now review the developed notation and how many of the well known fractals can be associated with it.
Firstly, we can say that the Cantor set results as an IFS attractor if $n=2$ and $P=1/3$.
As we saw in figure \ref{fr5}(a), the Sierpinski Triangle comes out when $n=3$ and $P=P(3,1)=1/2$ or $F_{\{3/1\}}$ and the Sierpinski Hexagon when $n=6$ and $P=P(6,2)=1/3$ or $F_{\{6/2\}}$.
The Greek Cross fractal appears when $n=4$ and $P=P(4,2)=1/2$ or $F_{\{4/2\}}$, while for $n=4$ and $P<1/2$ the invariant set for a Horseshoe map is produced.
The Sierpinski Pentagon appears for $n=5$ and $P=P(5,2)=1/(1+\text{golden ratio})$ or $F_{\{5/2\}}$.
When the centre map $S_c$ is taken into account, the Vicsec fractal can be produced when $n=4$ and $P=P_c=1/3$. 
Also, the Pentaflake and the Hexaflake are shown in figures \ref{fr11} and \ref{fr12} as $F_{\{5/2\}}[L^1,0]$ and $F_{\{6/2\}}[L^0,0]$, respectively.

However, the attractors of the special cases mentioned above do not originate from an unique star-polygons because 
if we generate random points on the attractor of the IFS, the image is defined by the maps $\{S_i\}$; see theorem \ref{th1}.
Therefore, if we do not alter $P$ and $n$, the IFS-attractors with any initial $\{n,m\}$-polygon will be the same.
Thus, we can assume that every attractor that originates from a \{2i,m\}-polygon is equivalent to the attractor of the IFS that originate from the \{2i,i\}-polygon, where the ratios and the rotations of the maps $\{S_1,...,S_{2i},S_c\}$ are kept the same as the ones used for the IFS of the \{2i,m\}-polygon.
However, when the exact polygons are plotted, as in section \ref{sec4}, due to the impossibility for infinite iterations to be realised, the integer $m$ also plays its role.
Two such examples are shown in figure \ref{fr21}, where in panels (a) and (b) the third iterations of the IFS of $F_{\{6/3\}}[L^2,\frac{\pi}{6}]$ and $F_{\{8/4\}}[L^2,\frac{\pi}{8}]$ are realised respectively.
One can see the difference with figures \ref{fr14}, \ref{fr15} and \ref{fr16} where $F_{\{6/2\}}[L^2,\frac{\pi}{6}]$, $F_{\{8/2\}}[L^2,\frac{\pi}{8}]$ and $F_{\{8/3\}}[L^2,\frac{\pi}{8}]$ are shown. 

In the case of odd $n=2i+1$ we do not have a star-polygon composed of $n$ line-segments that cross each other at the centre and are inscribed in all the \{2i+1,m\}-polygons in the same way as the \{2i,i\}-polygon is inscribed in any \{2i,m\}-polygon.
Therefore, we will construct such polygons as the star of $2i+1$ line segments that start from the centre of a $2i+1$ regular polygon and end up at its vertices. 
Let us denote this figure as $\{2i+1,i+1/2\}$-star polygon.
As the $\{2i+1,i+1/2\}$-polygon is inscribed in all the $\{2i+1,m\}$-polygons we can generalise that any attractor that originates from a \{n,m\}-polygon is equivalent to the attractor of the IFS that originates from the \{n,n/2\}-polygon where the ratios and the rotations of the maps $\{S_1,...,S_{n},S_c\}$ are kept the same as the ones used for the IFS of the \{n,m\}-polygon.
Two such examples are shown in figure \ref{fr21}, where in panels (c) and (d) the third iterations of the IFS of $F_{\{7/3.5\}}[M^1,\frac{\pi}{7}]$ and $F_{\{9/4.5\}}[M^1,\frac{\pi}{9}]$ are realised.
One can see the difference with figures \ref{fr19} where $F_{\{7/3\}}[M^1,\frac{\pi}{7}]$ was shown. 
\begin{Ex}\label{ex13}
 If the fractal figure \ref{fr21}(d) is infinitely iterated the attractor will have Hausdorff dimension:
$dim_HF_{\{9/4.5\}}[M^1,\frac{\pi}{9}]\approx1.8879$
\end{Ex}
Unlike the cases of $n=3,5,7$ and $9$, where the $F_{\{n/\frac{n}{2}\}}[M^1,\frac{\pi}{n}]$ is a non-self intersecting fractal with the copies of the initial $\{n/\frac{n}{2}\}$-polygon osculating with each other, for $n=2i+1$ when $n\geq11$, the scaled copies stop osculating and with the increase of the iterations they do not fill up the space in the most effective way.
The same effect appears when we take $F_{\{n/\frac{n}{2}\}}[L^2,\frac{\pi}{n}]$ for $n=2i$ when $n\geq10$.
Thus, the problem of finding scaling ratio for the $S_c$ for every $n\geq10$ where the rotation of $S_c$ is equal to $\pi/n$ stays as an open problem.
This is an important question due to the fact that the fractals that originate from an $n$-gon with rotation of $S_c$ equal to $\pi/n$ could have dimensions very close or equal to 2; see examples \ref{ex8} and \ref{ex9}.

\section*{Conclusion}

The present paper develops a universal technique that allows any star-polygon to be used for the construction of non-self intersecting fractal (Sierpinski $n$-gon or $n$-flake) by using IFS through random walk or through an exact scaling. 
Along the proposed scaling ratios, the Matlab code for IFS random walk fractal generation is provided so that someone interested in studying the geometry of this class of fractals could use it.
Important dimensions are computed, namely the dimension of the Sierpinki $\infty$-gon is proved to be 1, the dimension of the $\infty$-flake to be 2, and the dimension of $F\{6,2\}[L^2,\pi/6]$ to be 2 as well. 
It is also shown, that by using random walk IFS generator, identical attractors may result from different initial star-polygons. 
The proposed study can be extended if rotations are applied not only to the $S_c$ map, but to the $S_i$ maps as well.
However, this is still an ongoing research in development.

The techniques for construction and the provided ratios needed for the dimensions of the presented class of fractals is important not only for mathematicians, but also for engineers and other scientists who may be interested in fractal-shaped devices or who study the fractal shapes of nature.
With the advanced precision of the fabrication technology polyflakes may become an important design for devices such as antennas or chemical mixers for fuel cells, batteries, etc.
Also, fractal shapes are applicable in any kind of wave absorbers where the wave could be a sound wave, electromagnetic signal, light, caused by turbulent flow, etc.    
In other words, fractal designs are going to be part of the future physical devices at all scales and, hence, the research focused on fractal-shaped figures is important for many innovation processes.
  
\section*{Acknowledgement}
First of all I would like to thank my Geometry teacher Tanya Stoeva for the dedicated classes during the years at National Gymnasium of Natural Sciences and Mathematics "Lyubomir Chakalov" which still help me to overcome problems in the most intuitive way. Also, I would like to thank my family and friends for the support.  

\section*{Appendix (Matlab random generator)}\label{app}

\% name: Vassil Tzanov \\
\% function 'frac' that takes the matrix $C$, \\
\%$C=
\begin{bmatrix}
A1 & b1 &\\
A2 & b2 &\\
\vdots \\
An & bn &
\end{bmatrix}
$\\
\% Ai are matrices 2x2 ,bi are vectors 2x1, which are the i-th linear function in the "IFS"\\
\% $k$ defines the amount of points that we want to be plotted \\
\% on the attractor defined by $C$; the function 'frac' can plot fractals \\
\% with central polygon rotated at angle $rot$; the central map must be defined by\\
\% the last two rows of $C$; if there is no central map, $rot$ has to be defined as $0$\\
\texttt{function Fractal = frac(C,k,rot)}\\\\
\phantom{ss}\% verification of $C$\\
\texttt{\phantom{ss}D=size(C);\\
\phantom{ss}n=D(1)/2;\\
\phantom{ss}if D(2)$\sim$=3 \big| floor(n)$\sim$=n\\
\phantom{ssss}Fractal = 'Bad input size';\\
\phantom{ss}else}

\phantom{ssss}\% computation of the determinants of the matrices A\\
\texttt{\phantom{ssss}for i=1:n\\
\phantom{ssssss}dets(i)=abs(det(C((2*i-1):(2*i),1:2)));\\
\phantom{ssss}end\\
\phantom{ssss}dets=dets';}
  
\phantom{ssss}\% if some determinant is zero we have to add a small value, because\\
\phantom{ssss}\% we do not like this function to be executed with zero probability\\
\texttt{\phantom{ssss}dets=max(dets, max(dets)/(25*n));}

\phantom{ssss}\% the determinant are divided on their sum so we derive a probability vector\\
\texttt{\phantom{ssss}dets = dets/sum(dets);}
  
\phantom{ssss}\% the vector prob is defined as\\
\phantom{ssss}\% prob = [0, dets(1), dets(1)+dets(2), . . ., dets(1)+...+dets(n-1)] \\
\phantom{ssss}\% then "sum(prob$<$rand)" gives a random integer between $1$ and $n$ \\
\phantom{ssss}\% that we use in order to randomly choose one of the $n$-th functions\\
\texttt{\phantom{ssss}de=dets(1);\\
\phantom{ssss}prob(1)=0;\\
\phantom{ssss}for i=2:n\\
\phantom{ssssss}prob(i)=de;\\
\phantom{ssssss}de=de+dets(i);\\
\phantom{ssss}end}

\phantom{ssss}\% computations of the points' coordinates and filling the matrix $Fractal$\\
\texttt{\phantom{ssss}x=[0;0];\\
\phantom{ssss}Fractal=zeros(2,k+20);\\
\phantom{ssss}for j=1:(k+20)\\
\phantom{ssssss}int=sum(prob<rand);}\\
\phantom{ssssss}\% one of the functions is randomly chosen\\
\phantom{ssssss}i=int;\\
\phantom{ssssss}\% matrix computation of $t*x+(1-t)*y$, where $x$ is\\
\phantom{ssssss}\% the current point and $y$ is the point towards which $x$ is attracted\\
\phantom{ssssss}\% if $i$ corresponds to the last row, \\
\phantom{ssssss}\% a rotation at $rot$ radians around [0,0] must be executed\\
\texttt{\phantom{ssssss}if i==n\\
\phantom{ssssssss}x=[cos(rot),-sin(rot);sin(rot),cos(rot)]*x;\\
\phantom{ssssssss}x=C((2*i-1):(2*i),1:2)*x+(diag([1,1],0)-\\
-C((2*i-1):(2*i),1:2))*C((2*i-1):(2*i),3);\\
\phantom{ssssssss}Fractal(:,j)=x;\\
\phantom{ssssss}else\\
\phantom{ssssssss}x=C((2*i-1):(2*i),1:2)*x+(diag([1,1],0)-\\
-C((2*i-1):(2*i),1:2))*C((2*i-1):(2*i),3);\\ 
\phantom{ssssssss}Fractal(:,j)=x;\\
\phantom{ssssss}end\\
\phantom{ssss}end}
    
\phantom{ssss}\% the points of the matrix $Fractal$ are plotted\\ 
\texttt{\phantom{ssss}Fractal=Fractal(:,21:(k+20));\\
\phantom{ssss}plot(Fractal(1,:),Fractal(2,:),'k.','MarkerSize',1)\\
\phantom{ssss}axis('equal')\\
\phantom{ss}end\\
end }

\end{document}